\documentclass[12pt, a4paper,reqno]{amsart}
\usepackage{amsfonts,amsmath,amssymb,amsxtra,url,float} 
\allowdisplaybreaks[4] 
\usepackage[colorlinks,linkcolor=RoyalBlue,anchorcolor=Periwinkle,citecolor=Orange,urlcolor=black!55]{hyperref} 
\usepackage{color}
\usepackage[usenames,dvipsnames]{xcolor} 
\usepackage{enumitem}
\usepackage{geometry} 
\usepackage{rotating} 
\usepackage{lscape} 
\usepackage{multirow}
\usepackage{graphicx} 
\usepackage{subfigure} 
\usepackage{tikz}
\usetikzlibrary{decorations.pathreplacing,decorations.markings}
 \tikzset{
  on each segment/.style={
    decorate,
    decoration={
      show path construction,
      moveto code={},
      lineto code={
        \path [#1]
        (\tikzinputsegmentfirst) -- (\tikzinputsegmentlast);
      },
      curveto code={
        \path [#1] (\tikzinputsegmentfirst)
        .. controls
        (\tikzinputsegmentsupporta) and (\tikzinputsegmentsupportb)
        ..
        (\tikzinputsegmentlast);
      },
      closepath code={
        \path [#1]
        (\tikzinputsegmentfirst) -- (\tikzinputsegmentlast);
      },
    },
  },
  mid arrow/.style={postaction={decorate,decoration={
        markings,
        mark=at position 0.6 with {\arrow[#1]{stealth}} 
      }}},
}
\usetikzlibrary{arrows}
\usetikzlibrary{trees}

\usetikzlibrary{matrix}
\usetikzlibrary{patterns}
\usetikzlibrary{shadings} 
\usepackage[all]{xy} 
\usepackage{mathdots} 
\usepackage{dsfont} 
\usepackage{cite}
\usepackage{mathrsfs} 
\usepackage{multicol} 
\numberwithin{figure}{section}
\usepackage{marginnote} 
\usepackage{graphicx} 
\usepackage{multicol} 

\usepackage{fancyhdr} 
\newcommand{\checks}[1]{{\color{black}{#1}}} 

\newtheorem{theorem}{Theorem}[section]
\newtheorem{lemma}[theorem]{Lemma}
\newtheorem{corollary}[theorem]{Corollary}
\newtheorem{main theorem}[theorem]{Main Theorem}

\newtheorem{definition}[theorem]{Definition}

\newtheorem{remark}[theorem]{Remark}
\newtheorem{example}[theorem]{Example}

\newtheorem{question}[theorem]{Question}

\usetikzlibrary{arrows}

\numberwithin{equation}{section}

\def\<{\langle} 
\def\>{\rangle} 
\def\NN{\mathbb{N}} 
\def\RR{\mathbb{R}} 
\def\II{\mathbb{I}} 

\newcommand{\kk}{\mathds{k}} 
\newcommand{\Q}{\mathcal{Q}} 
\newcommand{\I}{\mathcal{I}} 
\newcommand{\Hom}{\mathrm{Hom}} %
\newcommand{\End}{\mathrm{End}} %
\newcommand{\sectcolor}{ }
\def\s{\mathfrak{s}}
\def\t{\mathfrak{t}}

\def\itLamb{\mathit{\Lambda}}
\def\im{\mathrm{Im}}

\def\bfS{\mathbf{S}}

\def\id{\mathbf{1}}
\def\scrN{\mathscr{N}\!\!or}
\def\scrA{\mathscr{A}}

\def\emb{\mathrm{emb}}
\def\dd{\mathrm{d}}

\def\iso{\omega}
\def\fct{\!{\lower 0.8ex\hbox{\tikz\draw (0pt, 0pt) node{\scriptsize$\mathfrak{F}$};}}\!}

\def\dfct{\varphi}
\def\isoinv{\varpi}
\def\emb{\mathfrak{e}}
\def\heart{{\color{red}\pmb{\heartsuit}}}
\def\diamo{{\color{red}\pmb{\diamondsuit}}}

\newcommand{\circled}{\ \lower-0.2ex\hbox{\tikz\draw (0pt, 0pt) circle (.1em);} \ }

\newcommand{\To}[1]{\mathop{-\!\!\!-\!\!\!-\!\!\!\longrightarrow}\limits^{#1}}
\newcommand{\oT}[1]{\mathop{\longleftarrow\!\!\!-\!\!\!-\!\!\!-}\limits_{#1}}

\newcommand{\w}[1]{\widehat{#1}}
\newcommand{\ol}[1]{\overline{#1}}
\newcommand{\mhomo}[1]{{{}_{#1}m}}
\newcommand{\Thomo}[1]{{{}_{#1}T}}
\newcommand{\sfct}[1]{{{}_{#1}E}}

\newcommand{\defines}{\it\color{black}}

\begin{document}

\def\headertitle{Normed modules over algebras II: Stieltjes integrations of on finite-dimensional algebras}

\title[\headertitle]{ Normed modules and The Stieltjes integrations of functions defined on finite-dimensional algebras }

\def\fstpage{1} 
\def\page{$\begin{matrix} {\color{white}0} \\ \thepage \end{matrix}$} 

\pagestyle{fancy}
\fancyhead[LO]{ }
\fancyhead[RO]{ }
\fancyhead[CO]{\ifthenelse{\value{page}=\fstpage}{\ }{\scriptsize{\headertitle}}}
\fancyhead[LE]{ }
\fancyhead[RE]{ }
\fancyhead[CE]{\scriptsize{ Hanpeng GAO, Shengda LIU, Yu-Zhe LIU, $\&$ Yucheng WANG}}
\fancyfoot[L]{ }
\fancyfoot[C]{\page} 
\fancyfoot[R]{ }
\renewcommand{\headrulewidth}{0.5pt} 

\thanks{$^{\ast}$Corresponding author.}


\author{Hanpeng Gao}
\address{H. Gao: School of Mathematical Sciences, Anhui University, Hefei 230601, Anhui, P. R. China}
\email{hpgao07@163.com;
{\it ORCID}: \href{https://orcid.org/0000-0001-7002-4153}{0000-0001-7002-4153}}

\author{Shengda Liu}
\address{S. Liu: The State Key Laboratory of Multimodal Artificial Intelligence Systems, Institute of Automation,
         Chinese Academy of Sciences, Beijing 100190, P. R. China.}
\email{thinksheng@foxmail.com;
{\it ORCID}: \href{https://orcid.org/0000-0003-1382-4212}{0000-0003-1382-4212}}

\author{Yu-Zhe Liu}
\address{Y.-Z. Liu: School of Mathematics and Statistics, Guizhou University, Guiyang 550025, Guizhou, P. R. China}
\email{liuyz@gzu.edu.cn / yzliu3@163.com;
{\it ORCID}: \href{https://orcid.org/0009-0005-1110-386X}{0009-0005-1110-386X}}

\author{Yucheng Wang}
\address{Y. Wang: Department of Mathematics, Nanjing University, Nanjing 210093, Jiangsu, P. R. China}
\email{wangyucheng2358@163.com;
{\it ORCID}: \href{https://orcid.org/0009-0003-4121-5294}{0009-0003-4121-5294}}






\maketitle

\vspace{-0.6cm}

\begin{abstract}
We define integrals for functions on finite-dimensional algebras, adapting methods from Leinster's research.
This paper discusses the relationships between the integrals of functions defined on subsets $\mathbb{I}_1 \subseteq \Lambda_1$ and $\mathbb{I}_2 \subseteq \Lambda_2$ of two finite-dimensional algebras,
under the influence of a mapping $\omega$, which can be an injection or a bijection.
We explore four specific cases:
\begin{itemize}
  \item $\omega$ as a monotone non-decreasing and right-continuous function;

  \item $\omega$ as an injective, absolutely continuous function;

  \item $\omega$ as a bijection;

  \item and $\omega$ as the identity on $\mathbb{R}$.
\end{itemize}
These scenarios correspond to the frameworks of Lebesgue-Stieltjes integration, Riemann-Stieltjes integration, substitution rules for Lebesgue integrals, and traditional Lebesgue or Riemann integration, respectively.

\end{abstract}

\fontsize{13pt}{15pt}\selectfont 

\tableofcontents 


\section{\sectcolor Introduction}

The Riemann-Stieltjes integrals, named after Bernhard Riemann and Thomas Joannes Stieltjes, were first articulated by Stieltjes in 1894 \cite{Stie1984}. These integrals act as foundational precursors to the Lebesgue integrations, which were introduced by Henri Lebesgue \cite{L1928}, and serve as invaluable tools in bridging the gap between discrete and continuous probability theories through the unification of statistical theorems. The development of the Lebesgue-Stieltjes integration further generalizes both the Riemann-Stieltjes and Lebesgue integrations within a broader measure-theoretic framework \cite{Carter2000}. This integral, employing the Lebesgue-Stieltjes measure, is particularly significant in applications across probability, stochastic processes, and other analytical branches, such as potential theory.

Typically, Riemann-Stieltjes and Lebesgue-Stieltjes integrals are utilized within Euclidean spaces and their subsets. However, when the integration domain lacks specific structural properties, such as an inner product or associative multiplication, a more generalized algebraic structure, such as finite-dimensional algebras, becomes necessary. This necessity prompts the redefinition of integration within these contexts and the development of suitable measures and analytical methods to investigate the properties of integrals. Given the significant influence of the properties of finite-dimensional algebras on this new type of integral, category theory becomes essential for dissecting the structure and properties of these algebras.

Introduced in the mid-20th century by Eilenberg and Lane \cite{Eil1945}, category theory provides a framework for understanding mathematical structures and their interrelations. In the early 2000s, Ehrhard and Regnier introduced differential calculus \cite{ER2003} and differential proof nets \cite{ER2006}, which formalized differentiation in linear logic \cite{G1987}, a form of logic introduced by Girard and modeled by symmetric monoidal categories. A few years later, Blute, Cockett, and Seely introduced differential categories \cite{BCS2006diff}, the appropriate categorical structure for modeling Ehrhard and Regnier's differential linear logic. This development prompted many analysts and computer scientists to conduct detailed research on differential categories, as seen in works such as \cite{BCS2015diff, Lemay2019, APL2021diff, IL2023ana, Lemay2023}.
During the same period, Cockett, Leinster, and Lemay considered algebraic or categorical interpretations of integrations using categories, as detailed in \cite{CL2018Cart-int,CL2018int,Lei2023FA}. In particular, Leinster provided a categorical interpretation for Lebesgue integrations through the category $\scrA^p$ ($p \geq 1$), the definition of which can be found in \cite{Lei2023FA} or Section \ref{subsec:cats} of this paper.


\checks{Let $\kk$ be a complete field contains some fully ordered subsets}.
In \cite{LLHZpre}, the authors point out that the integral of a function $f$ defined on a integration region $\II_{\itLamb} = [c,d]_{\kk}b_1\times \cdots \times [c,d]_{\kk}b_n$, a special subset of $\itLamb$, equals to $\w{T}(f)$,
where:
\begin{itemize}
  \item $[c,d]_{\kk}$ is a fully ordered subset of $\kk$;
  \item $\{b_1, \ldots, b_n\}$ is a $\kk$-basis of $\itLamb$ as an $n$-dimensional linear space;
  \item $\w{T}$ is a special $\itLamb$-homomorphism from the left $\itLamb$-module $\w{\bfS_{\tau}(\II_{\itLamb})}$, the completion of the set $\bfS_{\tau}(\II_{\itLamb})$ of all equivalence lasses of elementary simple functions (see Table \ref{tabel:notation}), to the $\itLamb$-module $\kk$ whose left $\itLamb$-action is defined as
      \begin{center}
        $\itLamb\times \kk \to \kk$

        $(a, k)\mapsto a\cdot k:=\tau(a)k;$
      \end{center}
  \item $\tau: \itLamb \to \kk$ is an epimorphism of $\kk$-algebras;
  \item and the left $\itLamb$-module structures of $\bfS_{\tau}(\II_{\itLamb})$ and $\kk$ are defined by $a\cdot f(x) \mapsto \tau(a)f(x)$ and $a\cdot k := \tau(a)k$ ($\forall a\in \itLamb$, $f(x)\in \bfS_{\tau}(\II_{\itLamb})$ and $k\in\kk$).
\end{itemize}

If $\scrA^p$ satisfies L-condition, that is, satisfies the conditions (L1) -- (L6) given in Theorem \ref{thm:LLHZpre-2} (2), then $\w{T}$, as a $\RR$-linear map, coincides with
the morphism given in \cite[Proposition 2.2]{Lei2023FA}, that is, $\w{T}$ sends any function $f$ in $\w{\bfS_{\tau}(\II_{\itLamb})}$ to its Lebesgue integral.

\checks{This paper provides a new explanation of Riemann-Stieltjes and Lebesgue-Stieltjes integrations through the lens of finite-dimensional algebras}, say $\itLamb$, and \checks{the categories $\scrN^p$ (see Subsection \ref{subsec:cats}) and $\scrA^p$}.
This interpretation comes from the following question considered by the second author of our paper.

\begin{question} \rm \label{quest}
Take $\itLamb_1$ and $\itLamb_2$ be two finite-dimensional $\kk$-algebras and $\iso$ the homomorphism of algebras.
Let, for any $r\in\{1,2\}$, $\scrA^1_r$ be the integral Banach module category of $\itLamb_r$,
and $\w{T}_r$, written as $(\scrA^1_r)\displaystyle\int_{\II_r}\cdot\dd\mu_r$, be the morphism in $\scrA^1_r$ which is induced by the $\itLamb_r$-homomorphism of $\itLamb_r$-modules used to describe the integral.
What are the relationships between $\w{T}_1 = (\scrA^1_1)\displaystyle\int_{\II_1}\cdot\dd\mu_1$ and $\w{T}_2 = (\scrA^1_2)\displaystyle\int_{\II_2}\cdot\dd\mu_2$?
\end{question}

Assume that $\iso:\itLamb_1\to\itLamb_2$ is either a {\defines measure preserving injection} or a {\defines measure preserving bijection},
where $\iso$ functions as both:
\begin{itemize}
  \item[(1)] for any subset $S_1$ of $\itLamb_1$ lying in $\Sigma(\II_1)$
    ($\Sigma(\II_1)$ is a $\sigma$-algebra defined by the integration region $\II_1$),
    the image $S_2=\mathrm{Im}(\iso|_{S_1})$ of the restriction $\iso|_{S_1}:S_1\to S_2$ of $\iso$ is also a subset of $\itLamb_2$ lying in $\Sigma(\II_2)$;
  \item[(2)] There exists a function $\fct:\RR \to \RR$ such that
    \begin{align*}
      (\fct\circled\mu_1)(S_1) = (\mu_2\circled\omega|_{S_1})(S_1)\ (=\mu_2(S_2))
    \end{align*}
    holds for all $S_1\in\Sigma(\II_1)$.
\end{itemize}
The primary aim of this paper is to address Question \ref{quest} and to present the following main results.

\begin{theorem}[{\rm Theorem \ref{thm:main1}}] \label{main result 1}
Let $f$ and $g$ respectively be two functions lying in $\w{\bfS_{\tau_1}(\II_1)}$ and $\w{\bfS_{\tau_2}(\II_2)}$ such that $g = \ol{\iso}(f) = f\circled\isoinv$,
where
\begin{itemize}
  \item[\rm(1)] $\w{\bfS_{\tau_1}(\II_1)}$ is the completions of the set $\bfS_{\tau_1}(\II_1)$ of all elementary simple functions defined on the integration region $ \displaystyle\II_1=\prod\limits_{i=1}^n [c_1,d_1]_{\kk}b_{1i}$ $(\subseteq \itLamb_1)$;
  \item[\rm(2)] $\w{\bfS_{\tau_1}(\II_2)}$ is the completions of the set $\bfS_{\tau_2}(\II_2)$ of all elementary simple functions defined on the integration region $ \displaystyle\II_2=\prod\limits_{i=1}^n [c_2,d_2]_{\kk}b_{2i}$ $(\subseteq \itLamb_2)$;
  \item[\rm(3)] $\isoinv$ is the left inverse of $\iso$.
\end{itemize}
Then there  exists a function $\fct:\RR\to\RR$ such that:
\begin{itemize}
  \item[\rm(a)]
    $\fct\circled\mu_1$ qualifies as a measure;
  \item[\rm(b)] The integral of $f$ over $\II_1$ with respect to $\fct\circled\mu_1$, equals the integral of $g$ over $\II_2$ with respect to $\mu_2$. That is
    \begin{align} \label{formula:Intro}
        (\scrA^1_1)\int_{\II_1} f \dd (\fct\circled\mu_1)
      = (\scrA^1_2)\int_{\II_2} g \dd\mu_2.
    \end{align}
\end{itemize}
\end{theorem}

Particularly, if $\iso$ functions as a measure-preserving isomorphism, its left inverse $\isoinv$ also preserves measure.
In this case, $\iso$ can be written as the standard form
\[ \itLamb_1
   \ni \sum_{i=1}^n k_ib_{1i}
   \
   \begin{smallmatrix}
   \To{\iso} \\
   \oT{\isoinv}
   \end{smallmatrix}
   \
   \sum_{i=1}^n k_ib_{2i}
   \in \itLamb_2. \]
Furthermore, if $c_1=c_2$ and $d_1=d_2$, there is no distinction between the measures $\mu_1$ and $\mu_2$. We obtain the following corollary.

\begin{corollary}[{\rm Corollary \ref{coro:main2}}]  \label{main result 2}
Let $f$ and $g$ be two functions lying in $\w{\bfS_{\tau_1}(\II_1)}$ and $\w{\bfS_{\tau_2}(\II_2)}$ such that $g=\ol{\iso}(f)$, respectively. If $\iso$ is an isomorphism, then
\[ (\scrA^1_1)\int_{\II_1} f \dd(\fct\circled\mu) = (\scrA^1_2)\int_{\II_2} g \dd\mu \]
\end{corollary}

Let $\itLamb_2=\RR$. In the last part of our paper, we provide four applications for
Theorem \ref{main result 1} and Corollary \ref{main result 2} as follows.
\begin{itemize}
  \item[(1)]
    In Subsection \ref{app:LSint}, we consider the scenario where $\iso$  is a monotone non-decreasing and right-continuous function.
    In this scenario, the function $\fct$, as specified by Theorem \ref{main result 1}, is identified as the identity $1_{\RR^{\ge0}}$. The left integral in formula \eqref{formula:Intro} may qualify as a Lebesgue-Stieltjes integral, provided that $\scrA^1_1$ and $\scrA^1_2$  meet certain conditions (see Example \ref{ex:LSint} for details).

  \item[(2)]
    In Subsection \ref{app:RSint}, we considered the case of $\iso$ to be an injective absolutely continuous function.
    In this case, $\fct=1_{\RR^{\ge0}}$, and the left integral of the formula (\ref{formula:Intro}) may be either a Riemann-Stieltjes integral
    if $\scrA^1_1$ and $\scrA^1_2$ satisfies some special conditions (see Example \ref{ex:RSint} for details).

  \item[(3)]
    In Subsection \ref{app:substitution}, we considered the case of $\fct=\iso$ to be a bijection,
    where $\fct:\itLamb_1=\RR \to \itLamb_2=\RR$ is a continuous function whose image is $[\alpha, \beta]$.
    In this case, if $\scrA^1_1 = \scrA^1_2 = \scrA^1$, then we obtain the substitution rules of Lebesgue integral, see (\ref{Substitution1}).

  \item[(4)]
    In Subsection \ref{app:Lint and Rint}, we considered the case of $\fct=\iso=\id_{\RR}$,
    where $\fct$ is the function of the form $\itLamb_1=\RR \to \itLamb_2=\RR, x\mapsto x$.
    In this case, if ${_{1}\II} = {_{2}\II} = [0,1]$ and $\scrA^1_1 = \scrA^1_2 = \scrA^1$,
    then (\ref{formula:Intro}), as a trivial situation, provide an interpretation for Lebesgue integration.

  \item[(5)]
    In application (3), it is well-known that all Riemann integrability functions defined on $[0,1]$ are elements lying in $\w{\bfS_{\id_{\RR}}([0,1])}$.
    Let $R([0,1])$, as a subset of $\w{\bfS_{\id_{\RR}}([0,1])}$, be the set of all Riemann integrability functions,
    then the restriction $\bigg(\displaystyle\int_{\II_1} \cdot\dd\mu_1\bigg)\bigg|_{R([0,1])}$
    of (\ref{formula:Intro}) provide an interpretation for Riemann integration.
\end{itemize}

This study utilizes category theory to rigorously define and analyze Stieltjes integrals within finite-dimensional algebras. Initially, we investigate how algebraic structures shape integration definitions, subsequently exploring the interactions between these integrals and the structural properties of algebras. The most profound insights are derived from integrating category theory concepts into the study of these integrals, providing a systematic approach to understanding and extending their applications in more generalized contexts. This method not only enriches theoretical understanding but also opens avenues for practical applications in complex systems involving algebraic structures.

We believe that  the contributions of this work are substantial, enhancing both the theoretical landscape and the practical applications of Stieltjes integrals:
\begin{enumerate}
  \item New Interpretation of Integral Theory: By analyzing Stieltjes integrals within the context of category theory, this offers a new perspective and interpretation for integral theory.
  \item Deep Connection Between Integrals and Algebraic Structures: We explore the relationship of integrals among different finite-dimensional algebraic structures, especially through the study of isomorphisms and homomorphisms, which has been less explored in previous research.
  \item Potential for New Applications: With a new understanding of Stieltjes integrals on finite-dimensional algebras, this research may open new pathways for applications of these integrals in fields such as physics, engineering, or computational science.
\end{enumerate}

The remainder of this paper is organized as follows.
Section \ref{sec:preliminaries} revisits essential preliminaries, delineating the constructs of normed and Banach modules, and introduces the categorical frameworks pertinent to our analyses.
In Section \ref{sect:3}, we rigorously examine the interrelations of integrals defined over distinct finite-dimensional algebras, specifically focusing on the mathematical consequences of various algebraic homomorphisms.
Section \ref{sect:app} is dedicated to the exposition of four practical scenarios where our theoretical constructs are applied, demonstrating their utility in both Lebesgue-Stieltjes and Riemann-Stieltjes integrations, as well as in establishing substitution rules for Lebesgue integrals. This section underscores the broader implications of our study, showcasing its potential to enrich the theoretical underpinnings and enhance the applicability of integral calculus in finite-dimensional algebraic settings.

\section{\sectcolor Preliminaries} \label{sec:preliminaries}

In this section, we recall some concepts in \cite{LLHZpre} whose originated from Leinster's works in \cite{Lei2023FA}.

\subsection{\sectcolor Normed modules and Banach modules}

Let $\kk$ be a field containing a fully ordered subset $\II$, written as $[c,d]_{\kk}$, whose minimal elements and maximal elements respectively are $c$ and $d$,
$\itLamb$ be a finite-dimensional $\kk$-algebra over $\kk$, and $B_{\itLamb} = \{b_i\mid 1\le i\le n\}$ be the basis of $\itLamb$ as a $\kk$-linear space.
We always assume that $[c,d]_{\kk}$ contains an element $\xi$ satisfying $c\prec \xi \prec d$ such that two order preserving bijections $\kappa_c: [c,d]_{\kk}\to [c,\xi]_{\kk}$ and $\kappa_d: [c,d]_{\kk}\to [\xi,d]_{\kk}$ exist.
We use the concepts and notations shown in Table \ref{tabel:notation} in our paper.

\begin{table}[htbp]
  \renewcommand{\arraystretch}{1.4}
  \centering
  \caption{Notations}
  \label{tabel:notation}
    \begin{tabular}{|p{3cm}|p{10cm}|}
      \hline \centering
        $\text{id}_S$
      & The identity map defined on a set $S$ \\
      \hline \centering
        $\id_S$
      & The function
        $\id_S(x)=
          {\begin{cases}
            1, & x\in S; \\
            0, & x\notin S
         \end{cases}}$ \\
      \hline \centering
        $\II_{\itLamb}$
      & The set $\sum_{i=1}^n \II b_i$ of all elements
        $\sum_{i=1}^nk_ib_i$ in $\itLamb$  satisfying $k_1, \ldots, k_n\in\II$
      \\ \hline \centering
        $(x,y)_{\kk}$
      & A fully ordered subset $\{k\in\II \mid x \prec   k\prec   y\}$ of $\II$
      \\ \hline \centering
        $(x,y]_{\kk}$
      & A fully ordered subset $\{k\in\II \mid x \prec   k\preceq y\}$ of $\II$
      \\ \hline \centering
        $[x,y)_{\kk}$
      & A fully ordered subset $\{k\in\II \mid x \preceq k\prec   y\}$ of $\II$
      \\ \hline \centering
        $[x,y]_{\kk}$
      & A fully ordered subset $\{k\in\II \mid x \preceq k\preceq y\}$ of $\II$
      \\ \hline \centering
        elementary simple function
        (defined on $\II_{\itLamb}$)
      & A map $f:\II_{\itLamb}\to\kk$ which is the finite sum
        $f(k_1, \ldots, k_n) = \sum_{i=1}^t k_i\id_{I_i}$ ($k_i\in\kk$)
        such that:
        {
        \begin{itemize}
          \item $I_j=\prod_{j=1}^n I_{ij}$, and, for any $1\le j\le n$, $I_{ij}$ is a subset of $\II$ which is one of the forms
              $(c_{ij},d_{ij})_{\kk}$, $(c_{ij},d_{ij}]_{\kk}$, $[c_{ij},d_{ij})_{\kk}$ and $[c_{ij},d_{ij}]_{\kk}$,
              where $c\preceq c_i \prec d_i \preceq d$;
          \item $\bigcup_{i=1}^t I_i = \II_{\itLamb}$, and $I_i\cap I_j=\varnothing$ for all $1\le i\ne j\le t$;
          \item $\id_{I_i}$ is the function $I_i \to \{1\}$.
        \end{itemize}
        }
      \\ \hline \centering
        $\bfS(\II_{\itLamb})$
      &  The set of all equivalence classes of elementary simple functions defined on $\II_{\itLamb}$
      \\ \hline \centering
         $\tau$
      &  A given homomorphism $\itLamb\to\kk$ between two $\kk$-algebras
      \\ \hline \centering
         $|\cdot|_{\kk}$ (=$|\cdot|$ for simplification)
      & A given norm defined on $\kk$
      \\ \hline \centering
        normed $\itLamb$-module
      &  A left $\itLamb$-module $M$ with a norm $\Vert\cdot\Vert: M \to \RR_{\ge 0}$ such that $\Vert am \Vert = |\tau(a)| \Vert m \Vert$ holds for all $a\in\itLamb$ and $m\in M$
      \\ \hline \centering
         $\Vert\cdot\Vert_M$ (=$\Vert\cdot\Vert$ for simplification)
      &  The norm defined on $\itLamb$-module $M$
      \\ \hline \centering
         $\displaystyle(\mathrm{B}/\mathrm{L}/\mathrm{R})\int$
      & Bochner/Lebesgue/Riemann integration
      \\ \hline \centering
         $\displaystyle(\mathrm{L}\text{-}\mathrm{S})\int$
      & Lebesgue-Stieltjes integration
      \\ \hline \centering
         $\displaystyle(\mathrm{R}\text{-}\mathrm{S})\int$
      & Riemann-Stieltjes integration
      \\ \hline
    \end{tabular}
\end{table}

It is clear that $\bfS(\II_{\itLamb})$ is an infinite-dimensional $\kk$-linear space, and it is a $\itLamb$-module defined by the following $\kk$-linear map
\begin{center}
  $\theta:\itLamb \to \End_{\kk}(\bfS(\II_{\itLamb}))$, $a\mapsto \varphi_a$,
\end{center}
where $\varphi_a: \bfS(\II_{\itLamb}) \to \bfS(\II_{\itLamb})$ sends any function $f$ in $\bfS(\II_{\itLamb})$ to the function $\varphi_a(f):=\tau(a)f$, that is, for any $a\in\itLamb$ and $f\in \bfS(\II_{\itLamb})$, the left $\itLamb$-action is
\begin{center}
  $\itLamb\times \bfS(\II_{\itLamb}) \to \bfS(\II_{\itLamb})$, $a\cdot f = \varphi_a(f) = \tau(a)f$.
\end{center}

Let $\Sigma(\II)$ be the $\sigma$-algebra generated by
\begin{center}
  $\{(c',d')_{\kk}, [c',d')_{\kk}, (c',d']_{\kk}, [c',d']_{\kk} \mid c\preceq c'\preceq d'\preceq d\}$,
\end{center}
and $\mu_{\II}$ be the measure defined on $\Sigma(\II)$ such that, for any $a\in\II$, $\mu_{\II}(\{a\})$ is zero.
Let $\Sigma(\II_{\itLamb})$ be the $\sigma$-algebra generated by all subsets $\sum_{i=1}^n I_ib_i$ of $\itLamb$, where, for any $1\le i\le n$, $I_i$ is one of the form $(c_i,d_i)_{\kk}$, $(c_i,d_i]_{\kk}$, $[c_i,d_i)_{\kk}$ and $[c_i,d_i]_{\kk}$, $c\preceq c_i \preceq d_i \preceq d$.
Obviously, $\mu_{\II}$ induces a measure, say $\mu_{\II_{\itLamb}}$ (or, without ambiguity, say $\mu$) in our paper, defined on $\Sigma(\II_{\itLamb})$, such that
\begin{align} 
  \mu_{\II_{\itLamb}}\bigg(\sum_{i=1}^n I_ib_i\bigg) = \prod_{i=1}^n \mu_{\II}(I_i). \nonumber
\end{align}
\checks{Now, we can define the equivalence classes of functions by
\begin{center}
  $f_1 \sim f_2$ if and only if $\mu(\{x\in\II_{\itLamb} \mid f_1(x)\ne f_2(x)\})=0$.
\end{center}
If we do not differentiate between two functions which are equivalent,}
then the map $\Vert\cdot\Vert_p: \itLamb \to \RR_{\ge 0}$ defined as
\[\Vert k_1b_1+\cdots+k_nb_n \Vert_p = \bigg(\sum_{i=1}^{n} |k_i|^p\bigg)^{\frac{1}{p}}\]
is a norm of $\itLamb$ such that $\itLamb$, as a left $\itLamb$-module defined by
\[\itLamb \times \itLamb \to \itLamb,\ (a_1,a_2) \mapsto \tau(a_1)a_2,\]
is a  norm $\itLamb$-module, and the map $\Vert\cdot\Vert:\bfS(\II_{\itLamb}) \to \RR_{\ge 0}$ defined as
\[\Vert f\Vert \mapsto \bigg(\sum_{i=1}^t (|k_i|\mu(I_i))^p\bigg)^{\frac{1}{p}}\]
is a norm \checks{defined on} $\bfS(\II_{\itLamb})$ such that $\bfS(\II_{\itLamb})$ is a normed $\itLamb$-module which is written as $\bfS_{\tau}(\II_{\itLamb})$ in this paper.

We can define the {\defines completion} of normed $\itLamb$-module $M$ by the following steps.
\begin{itemize}
  \item[Step 1] Define the {\defines Cauchy sequences}, that is, the sequence $\pmb{x}=\{x_i\}_{i=1}^{+\infty}$ in $M$ such that, for any $\varepsilon>0$, we can find an integer $N\in\NN$ such that $\Vert x_u-x_v \Vert < \varepsilon$ holds for all $u,v>N$.

  \item[Step 2]
    For any Cauchy sequence $\pmb{x}=\{x_i\}_{i=1}^{+\infty}$ in $M$, define the equivalent class, say $[\pmb{x}]$, containing $\pmb{x}$ is the set of all Cauchy sequences $\pmb{y}=\{y_i\}_{i=1}^{+\infty}$ in $M$ which satisfies that, for arbitrary $\varepsilon>0$, there exists $N\in\NN$ such that $\Vert x_i-y_i\Vert<\varepsilon$ holds for all $i>N$.

  \item[Step 3]
    Let $\mathfrak{C}(M)$ be the set of all Cauchy sequences in $M$. Then it is a $\itLamb$-module, and $[\{0\}]$ is a submodule of it, where $\{0\}=0,0,\ldots$.
    Then the quotient module $\mathfrak{C}(M)/[\{0\}]$, written as $\w{M}$, is a normed $\itLamb$-module whose norm $\Vert\cdot\Vert_{\w{M}}$ can bd defined by the norm $\Vert\cdot\Vert_M$ of $M$ as follows.
    \[ \Vert \pmb{x} \Vert_{\w{M}} = \Vert c \Vert, \text{ if } \pmb{x}=\{x_i\}_{i=1}^{+\infty} \text{ is equivalent to }  \{c\}_{i=1}^{+\infty} = c,c,\ldots. \]
\end{itemize}
By the first isomorphism theorem, $\w{M}$ is isomorphic to the set of all equivalence classes of Cauchy sequences which as a $\itLamb$-module.
Furthermore, a normed $\itLamb$-module is called a {\defines Banach $\itLamb$-module} if its completeness is equal to itself.

\subsection{\sectcolor The categories $\scrN^p$ and $\scrA^p$ and their special objects} \label{subsec:cats}

In this section we review the definitions of the categories $\scrN^p$ and $\scrA^p$, and recall some results in \cite{LLHZpre}.

\subsubsection{\sectcolor The categories $\scrN^p$ and $\scrA^p$}
Assume that $\itLamb$  be an arbitrary finite-dimensional $\kk$-algebra whose dimension $\dim_{\kk}\itLamb$ is $n$.

\begin{definition}\rm
The category $\scrN^p_{\itLamb}$ (without ambiguity, we write it as $\scrN^p$), say {\defines the integral normed module category of $\itLamb$}, is defined as follows.
\begin{itemize}
  \item The object in $\scrN^p$ is a triple $(N, v, \delta)$ of normed $\itLamb$-module $N$, an element $v$ lying in $N$ and a $\itLamb$-homomorphism $\delta: N^{\oplus_p 2^n} \to N$ satisfying $(v,\cdots, v) \mapsto v$
      such that, for any monotone decreasing Cauchy sequence $\{x_i\}_{i\in\NN}$ in $N$ (the partial order ``$\preceq$'' in this case is defined as, for any $x',x''\in N$, $x'\preceq x''$ if and only if $\Vert x'-\underleftarrow{\lim}x_i \Vert \le \Vert x''-\underleftarrow{\lim}x_i \Vert$), the commutativity
      \[\delta(\underleftarrow{\lim}x_i) = \underleftarrow{\lim}\delta(x_i)\]
      of the inverse limit and the $\itLamb$-homomorphism holds,
      where, for any $2^n$ $\itLamb$-modules $X_i$, $\ldots$, $X_{2^n}$,
      \[{\bigoplus\limits_{i=1}^{2^n}{}_p}\ X_i = X_1\oplus_p \cdots \oplus_p X_{2^n}\]
      is the direct sum $\bigoplus\limits_{i=1}^{2^n}X_i$ whose norm is
      \[\Vert\cdot\Vert: \bigoplus\limits_{i=1}^{2^n}X_i \to \RR_{\ge0}\]
      \[ (x_1,\cdots, x_{2^n}) \mapsto
        \bigg(
          \bigg(
            \frac{\mu_{\II}(\II)}{\mu(\II_A)}
          \bigg)^n
          \sum_{i=1}^{2^n} \Vert x_i \Vert^p
        \bigg)^{\frac{1}{p}}. \]
  \item The morphism $(N, v, \delta) \to (N',v',\delta')$ is a $\itLamb$-homomorphism $f: N\to N$ such that the following conditions hold.
      \begin{itemize}
        \item $f(v)=v'$
        \item the diagram
          \[ \xymatrix{
           N^{\oplus_p 2^n} \ar[d]_{f^{\oplus 2^n}} \ar[r]^{\delta} & N \ar[d]^{f} \\
           N'^{\oplus_p 2^n} \ar[r]_{\delta'} & N'
          } \]
          commutes.
    \end{itemize}
\end{itemize}
\end{definition}

\begin{definition} \rm
The category $\scrA^p_{\itLamb}$ (without ambiguity, we write it as $\scrA^p$), say {\defines the integral Banach module category of $\itLamb$}, is the full subcategory of $\scrN^p$ such that each object $(N,v,\delta)$ in $\scrA^p$ satisfies that $N$ is a Banach $\itLamb$-module.
\end{definition}

\subsubsection{\sectcolor The objects $(\bfS_{\tau}(\II_{\itLamb}), \id_{\II_{\itLamb}}, \gamma_{\xi})$ and $(\kk, \mu(\II_{\itLamb}), m)$}

Take $\gamma_{\xi}$, say {\defines juxtaposition map}, is the $\kk$-linear map
\[\gamma_{\xi}: \bfS_{\tau}(\II)^{\oplus_{p} 2^n} \to \bfS_{\tau}(\II)\]
defined by
\[\gamma_{\xi}(\pmb{f}) (k_1,\ldots,k_n)
 = \sum_{(\delta_1, \ldots, \delta_n)\in \{c,d\}\times\cdots\times\{c,d\}}
   \id_{\pmb{\kappa}}\cdot f_{(\delta_1, \ldots, \delta_n)} (\kappa_{\delta_1}^{-1}(k_1),\ldots,\kappa_{\delta_n}^{-1}(k_n)),\]
where $\pmb{\kappa} = \kappa_{\delta_1}(\II)\times \cdots\times \kappa_{\delta_n}(\II)$; and $k_1\ne\xi$, $\ldots$, $k_n\ne\xi$.
Then $(\bfS_{\tau}(\II_{\itLamb}), \id_{\II_{\itLamb}}, \gamma_{\xi})$ is an object in $\scrN^p$,
its completion $(\w{\bfS_{\tau}(\II_{\itLamb})}, \id_{\II_{\itLamb}}, \w{\gamma}_{\xi})$ is an initial object in $\scrA^p$.
It follows that, for any object $(N,v,\delta)$ in $\scrA^p$, the morphism space
\begin{center}
  $\Hom_{\scrN^p}((\bfS_{\tau}(\II_{\itLamb}), \id_{\II_{\itLamb}}, \gamma_{\xi}), (N,v,\delta))$
\end{center}
contains only one morphism. Thus, we have the following result.

\begin{theorem} \label{thm:LLHZpre-1}{\rm(\!\cite[Theorems 6.3]{LLHZpre})}
For any object $(N,v,\delta)$ in $\scrA^p$, there is a unique morphism
\[T_{(N,v,\delta)} \in \Hom_{\scrN^p}((\bfS_{\tau}(\II_{\itLamb}), \id_{\II_{\itLamb}}, \gamma_{\xi}), (N,v,\delta)).\]
can be extended to the morphism
\[\w{T}_{(N,v,\delta)} \in \Hom_{\scrA^p}((\w{\bfS_{\tau}(\II_{\itLamb})}, \id_{\II_{\itLamb}}, \w{\gamma}_{\xi}), (N,v,\delta)). \]
\end{theorem}

The field $\kk$ can be seen as a $\itLamb$-module with $\itLamb\times\kk\to\kk, (a,k)\mapsto a\cdot k:=\tau(a)k$.
Then, the norm $|\cdot|:\kk\to\RR_{\ge 0}$ defined on $\kk$ satisfies that
\[ |a\cdot k| = |\tau(a)k| = |\tau(a)||k|. \]
Thus, $\kk$ is a normed $\itLamb$-module.

Now, we denote $\xi\in\II=[c,d]_{\kk}$ by $\xi_{11}$. The element $\xi_{11}$ divides $\II=:\II^{(01)}$ to two subsets $[a,\xi_{11}]_{\kk}=:\II^{(11)}$ and $[\xi_{11},b]_{\kk}=:\II^{(12)}$.
Next, let $\xi_{22}=\xi_{11}$ $(=\xi)$, and denote by $\xi_{21}$ and $\xi_{23}$ the two elements in $\II_{\itLamb}$ such that
\begin{itemize}
  \item $c\prec\xi_{21}=\kappa_c\kappa_c(d) = \kappa_c\kappa_d(c) = \kappa_d\kappa_c(c) = \kappa_c(\xi_{11})\prec\xi_{22}$;
  \item $\xi_{22}\prec\xi_{23}=\kappa_d\kappa_d(c)=\kappa_d\kappa_c(d)=\kappa_d\kappa_c(d) = \kappa_d{\xi_{11}}\prec d$.
\end{itemize}
Then $\II$ is divided to four subsets which are of the form
$\II^{(2t)} = [\xi_{2t}, \xi_{2\ t+1}]_{\kk}$ ($0\le t\le 3$)
by $c=\xi_{20} \prec \xi_{21} \prec \xi_{22} \prec \xi_{23} \prec \xi_{24}=d$.
Repeating the above step $t$ times, we obtain a sequence of $2^t-1$ elements lying in $\II$
\[c=\xi_{t0} \prec \xi_{t1} \prec \xi_{t2} \prec \cdots \prec \xi_{t2^t}=d, \]
where all $2^t$ subsets which are of the form $\II^{(t\ s+1)}=[\xi_{ts}, \xi_{t\ s+1}]_{\kk}$,
and obtain $2^t$ order preserving bijections $\kappa_{\xi_{ts}}$ from $\II^{(t\ s+1)}$ to $\II^{(01)}$.

For any family of subsets $(\II^{(u_iv_i)})_{1\le i\le n}$ ($1\le v_i\le 2^{u_i}$),
we denote by $\id_{(u_iv_i)_i}$ the function $\id_{\II_{\itLamb}}\big|{}_{\prod_{i=1}^{n}\II^{(u_iv_i)}}$ for simplification,
where $\II^{(u_iv_i)} \cong \II^{(u_iv_i)}\times\{b_i\}\subseteq \II_{\itLamb}$ holds for all $i$
and $B_{\itLamb}=\{b_i \mid 1\le i\le n\}$ is the $\kk$-basis of $\itLamb$.

\begin{definition}[{Step function spaces}] \rm \label{def:sfct.spaces}
Let $E_{u}$ be the set of all functions which are of the form
\[\sum_{(u_iv_i)_i} k_{(u_iv_i)_i}\id_{(u_iv_i)_i} \ (= \sum_i k_i\id_{I_i} \text{ for simplification}), \]
where each $k_{(u_iv_i)_i}$ lies in $\kk$, the number of all summands of the above sum is $(2^u)^n=2^{un}$,
and each $(u_iv_i)_i$ corresponds to the Cartesian product $\prod_{i=1}^{n}\II^{(u_iv_i)}$.
The set $E_u$ is called a {\defines step function spaces {\rm(}defined on $\II_{\itLamb}${\rm)}} and each function lying in $E_u$ is called a {\defines step functions},
Then it is easy to check that each $E_u$ is a normed submodule of $\bfS_{\tau}(\II_{\itLamb})$,
and $E_{u}\subseteq E_{u+1}$ because each step function constant on each of $\II^{(uv)}$
is equivalent to a step function constant on each of $\II^{(u+1\ v)}$. Thus,
\[\kk \cong E_{0} \subseteq E_{1} \subseteq \ldots \subseteq E_{t} \subseteq \ldots
\subseteq \bfS_{\tau}(\II_{\itLamb}) \subseteq \w{\bfS_{\tau}(\II_{\itLamb})}.\]
The authors of \cite{LLHZpre} have checked that
\begin{align}\label{limEu=S}
  \underrightarrow{\lim}E_u \cong \w{\bfS_{\tau}(\II_{\itLamb})},
\end{align}
see \cite[Lemma 5.4]{LLHZpre} or cf.  \cite[Examples 2.2.4 (h) and 2.2.6 (g)]{Bor1994}.
\end{definition}

For any $u\in \NN$, we define
\begin{align}\label{def:Tu}
 T_u: E_u \to \kk, \ \sum_i k_i\id_{I_i} \mapsto \sum_i k_i\mu(I_i).
\end{align}
Notice that $E_u$ is a $\itLamb$-module defined as
\[\itLamb\times E_u \to E_u, (a,f)\mapsto a\cdot f:= \tau(a)f, \]
then it is easy to see that
\[T_u(a\cdot f) = T_u(\tau(a)f)=\tau(a)T_u(f)=a\cdot T_u(f),\]
that is, $T_u$ is a $\itLamb$-homomorphism.
The restriction $\gamma_{\xi}|_{E_u^{\oplus_p 2^n}}$ of $\gamma_{\xi}$, written as $\gamma_{\xi}$ for simplification, and the $\itLamb$-homomorphism $T_u$ induces a map
\[m_u:\kk^{\oplus_p 2^n} \to \kk\]
by the following way. For any $k\in\kk$, consider the function $ \displaystyle f_k = \frac{k}{\mu(\II_{\itLamb})}\id_{\II_{\itLamb}}$,
we have
\[T_u(f_k) = T_u\Big(\frac{k}{\mu(\II_{\itLamb})} \id_{\II_{\itLamb}}\Big)
          = \frac{k}{\mu(\II_{\itLamb})} T_u(\id_{\II_{\itLamb}}) = k.\]
Then for any $\pmb{k}=(k_1,\ldots, k_{2^n})\in \kk^{\oplus_p 2^n}$, $\pmb{f}_{\pmb{k}} = (f_{k_1},\ldots,f_{k_{2^n}})\in E_u^{\oplus_p 2^n}$ is a preimage of $\pmb{k}$ under the $\itLamb$-homomorphism $T_u^{\oplus 2^n}$.
We define $m_u(\pmb{k}) = T_{u+1}(\gamma_{\xi}(\pmb{f}_{\pmb{k}}))$, that is, $m_u$ is a $\itLamb$-homomorphism such that the following diagram
\begin{align}
\xymatrix{
  E_u^{\oplus_p 2^n} \ar[r]^{\gamma_{\xi}} \ar[d]_{T_u^{\oplus 2^n}}
& E_{u+1} \ar[d]^{T_{u+1}} \\
  \kk^{\oplus_p 2^n} \ar[r]_{m_u}
& \kk
} \nonumber
\end{align}
commutes, cf.  \cite[Lemma 7.2]{LLHZpre}. Then, consider the direct system
\begin{align}\label{formula:dir.sys}
  ((E_u)_{u\in\NN}, (E_{u_1} \mathop{\to}\limits^{\subseteq} E_{u_2})_{u_1\le u_2})
\end{align}
$\{m_u\}_{u\in\NN}$ induces a map
\begin{align}\label{def:mhomo}
  m=\underrightarrow{\lim}\ m_u : \kk^{\oplus_p 2^n} \to \kk
\end{align}
which satisfies that $m(\mu(\II_{\itLamb}),\ldots, \mu(\II_{\itLamb}))=\mu(\II_{\itLamb})$ since, for each $u\in\NN$, $m_u$ sends $(\mu(\II_{\itLamb}),\ldots,\mu(\II_{\itLamb}))$ to $\mu(\II_{\itLamb})$.
Furthermore, $(\kk,\mu(\II_{\itLamb}),m)$ is an object in $\scrN^p$.

\subsubsection{\sectcolor The categorization of Lebesgue integrations}

The following result show that the morphism $T_{(N,v,\delta)}$ given in Theorem \ref{thm:LLHZpre-1} provide a categorization of Lebesgue integrations in the case for $(N,v,\delta)=(\kk,\mu(\II_{\itLamb}),m)$ if $\kk$ is a completion field.

\begin{theorem} \label{thm:LLHZpre-2}
Formalizing the categorical interpretation of Lebesgue integration as follows:
\begin{itemize}
  \item[{\rm(1)}] {\rm(\!\!\cite[Theorems 7.6]{LLHZpre})}
    If $\kk$ is a completion field which is an extension of $\RR$ and $p=1$, then there is an object, which is of the form $(\kk,\mu(\II_{\itLamb}),m)$, in $\scrA^1$ such that
    \[ \w{T}_{(\kk,\mu(\II_{\itLamb}),m)}: f \mapsto \w{T}_{(\kk,\mu(\II_{\itLamb}),m)}(f),
    \text{ denote it by } \ (\scrA^1)\int_{\II_{\itLamb}}\cdot\dd\mu\]
    is a unique morphism in $\Hom_{\scrA^1}((\bfS_{\tau}(\II_{\itLamb}),\id_{\II_{\itLamb}}, \gamma_{\xi}), (\kk,1,m))$ satisfying following statements.
    \begin{itemize}
      \item  $\w{T}_{(\kk,\mu(\II_{\itLamb}),m)}(\id_{\II_{\itLamb}}) = \mu(\II_{\itLamb})$;
      \item  $\w{T}_{(\kk,\mu(\II_{\itLamb}),m)}: \bfS_{\tau}(\II_{\itLamb}) \to \kk$ is a homomorphism of $\itLamb$-modules;
      \item  $\w{T}_{(\kk,\mu(\II_{\itLamb}),m)}(|f|) \le |\w{T}_{(\kk,\mu(\II_{\itLamb}),m)}(f)|$.
    \end{itemize}

  \item[{\rm(2)}] {\rm(cf.  \cite[Proposition 2.2]{Lei2023FA} or \cite[Corollary 8.3]{LLHZpre})}
    If $\scrA^p$ satisfies {\defines L-conditions}, that is, it is such that the following conditions hold:
    \begin{itemize}
      \item[{\rm(L1)}] $p=1$;
      \item[{\rm(L2)}] $\kk=\RR=\itLamb$ $($in this case, the norms of $\kk$ and $\itLamb$ coincide$)$, or $\kk=\mathbb{C}=\itLamb$ $($in this case, the norm of $\mathbb{C}$ is defined by the modulus of complex numbers$)$;
      \item[{\rm(L3)}] $\II=[x_1,x_2]$ $(=\II_{\RR})$;
      \item[{\rm(L4)}] $\displaystyle\xi=\frac{x_1+x_2}{2}$, $\displaystyle\kappa_{x_1}(x)=\frac{x+x_1}{2}$, $\displaystyle\kappa_{x_2}(x)=\frac{x+x_2}{2}$;
      \item[{\rm(L5)}] $\tau=\mathrm{id}_{\kk}$;
      \item[{\rm(L6)}] $\mu: \Sigma(\II_{\itLamb}) \to \RR^{\ge 0}$ is the Lebesgue measure defined on the $\sigma$-algebra $\Sigma(\II_{\itLamb})$,
    \end{itemize}
    then, there exists an object in $\scrA^1$, which is of the form $(\kk,1,m)$, such that for any $f\in\w{\bfS_{\tau}(\II_{\itLamb})}$,
    $\w{T}_{(\kk,1,m)}$ is a $\kk$-linear space from $\w{\bfS_{\tau}(\II_{\itLamb})}$ to $\kk$
    sending any $f\in\w{\bfS_{\tau}(\II_{\itLamb})}$ to its Lebesgue integral, i.e.,
    \[ \w{T}_{(\kk,1,m)}(f) = \mathrm{(L)}\checks{\int_{x_1}^{x_2}} f\dd\mu. \]
\end{itemize}
\end{theorem}

The morphism $\w{T}_{(\kk,\mu(\II_{\itLamb}),m)}$ in Theorem \ref{thm:LLHZpre-2} equals to the direct limit $\underrightarrow{\lim}T_u$, where the direct system is
$(T_u, (\emb_{u_1u_2}: T_{u_1}\to T_{u_2})_{u_1\le u_2})$ which is induced by the direct system (\ref{formula:dir.sys}),
and the morphism $\emb_{u_1u_2}$, say an extension of the embedding of two $\itLamb$-homomorphisms $T_{u_1}$ and $T_{u_2}$, given by the restrictions $T_{u_2}|_{E_{u_1}}=T_{u_1}$
and $(\int_{\II_{\itLamb}}\cdot\dd\mu)|_{E_u} = T_{(\kk,\mu(\II_{\itLamb}),m)}|_{E_u} = T_u$.

The {\defines partial order ``$\preceq_f$'' (with respect to $f$)} of $\w{\bfS_{\tau}(\II_{\itLamb})}$ can be defined by
\[g \preceq_f h  \text{ if and only if }
    \bigg\Vert (\scrA^1)\int_{\II_{\itLamb}} (g-f) \dd\mu \bigg\Vert
\le \bigg\Vert (\scrA^1)\int_{\II_{\itLamb}} (h-f) \dd\mu \bigg\Vert. \]
By (\ref{limEu=S}), for any $f\in \w{\bfS_{\tau}(\II_{\itLamb})}$, we can, in the sense of the partial ordered ``$\preceq_f$'', find a monotone decreasing Cauchy sequence $\{f_{i}: E_{u_i}\to \kk\}_{i\in\NN}$ of step functions in $\bigcup_{u\in\NN} E_u$ such that $\underleftarrow{\lim}f_i = f$,
and since $\int_{\II_{\itLamb}}\cdot\dd\mu$ is an extension of $T_u$ satisfying
\[T_u = T_{(\kk,\mu(\II_{\itLamb}),m)}|_{E_u} = \bigg((\scrA^1)\int_{\II_{\itLamb}}\cdot\dd\mu\bigg)\bigg|_{E_u},\]
we have
\begin{align} \label{analy.lim}
  (\scrA^1)\int_{\II_{\itLamb}}f\dd\mu
= \w{T}_{(\kk,\mu(\II_{\itLamb}),m)}(\underleftarrow{\lim} f_i)
\mathop{=}\limits^{\spadesuit} \underleftarrow{\lim} T_{u_i}(f_i)
\ \Big(\mathop{=}\limits^{\clubsuit}  \lim_{i\to+\infty} T_{u_i}(f_i)\Big),
\end{align}
where the notation ``$\lim\limits_{i\to+\infty}$'' is the limit in analysis because all $T_{u_i}(f_i)$ are elements in the field $\kk$;
$\spadesuit$ holds since $\{|T_{u_i}(f_i-f)|\}_{i\in \NN}$ is a monotone decreasing Cauchy sequence (for $i\ge j$, $f_i\preceq_f f_j$ yields $|T_{u_i}(f_i-f)| \le |T_{u_j}(f_j-f)|$);
and $\clubsuit$ holds since $\{|T_{u_i}(f_i-f)|\}_{i\in \NN}$ is a Cauchy sequence in $\RR$.

\section{\sectcolor The relationships of two integrals on $\itLamb_1$ and $\itLamb_2$} \label{sect:3}

We always assume that $\kk$ is a complete field, the norm defined on $\kk$ is written as $|\cdot|$, containing fully ordered subset $\II=[c,d]_{\kk}$ such that there is an element $\xi\in(c,d)_{\kk}$ providing two order preserving bijections
\begin{center}
$\kappa_c: [c,d]_{\kk}\to [c,\xi]_{\kk}$ and $\kappa_d: [c,d]_{\kk}\to [\xi,d]_{\kk}$.
\end{center}
Then, for any finite-dimensional $\kk$-algebra $\itLamb$ and a given homomorphism $\tau:\itLamb\to\kk$ of algebras,
we can define $\II_{\itLamb}$, the $\sigma$-algebra $\Sigma(\II_{\itLamb})$, the measure $\mu_{\II_{\itLamb}}$ (=$\mu$ for simplification), $\bfS_{\tau}(\II_{\itLamb})$ (resp. normed $\itLamb$-module), $\w{\bfS_{\tau}(\II_{\itLamb})}$ (resp. Banach $\itLamb$-module), and two categories $\scrN^p_{\itLamb}$ and $\scrA^p_{\itLamb}$.
Furthermore, we can provide a categorization of Lebesgue integrations by the morphisms $T_{(\kk, \mu(\II_{\itLamb}), m)}$ and $\w{T}_{(\kk, \mu(\II_{\itLamb}), m)} = \displaystyle\int_{\II_{\itLamb}} \cdot \dd\mu$.
Naturally, we obtain the following questions.

\begin{question} [Question \ref{quest}] \rm \
For two finite-dimensional $\kk$-algebras $\itLamb_1$ and $\itLamb_2$ and two homomorphisms of algebras $\tau_1:\itLamb_1\to\kk$ and $\tau_2:\itLamb_2\to\kk$,
we can provide the definitions of integration $ \displaystyle\int_{\II_{\itLamb_1}}\cdot\dd\mu_{\II_{\itLamb_1}}$ and $ \displaystyle\int_{\II_{\itLamb_2}}\cdot\dd\mu_{\II_{\itLamb_2}}$ accordingly by similar way.
If $\itLamb_1 \cong \itLamb_2$, what is the connection between the two types of integrations mentioned above?
\end{question}

We will answer the above questions in this section.

\subsection{\sectcolor  Some lemmas}

Let $\itLamb_r$, $r\in\{1,2\}$, be two finite-dimensional $\kk$-algebras with $\dim_{\kk}\itLamb_r=n_r$,
and $B_r = \{b_{r1},\ldots,b_{rn_r}\}$ be the basis of $\itLamb_r$.
Let $\iso:\itLamb_1\to\itLamb_2$ be a map between two finite-dimensional $\kk$-algebras $\itLamb_1$ and $\itLamb_2$ in this section.

\subsubsection{\sectcolor The corresponding between $\bfS_{\tau_1}(\II_{\itLamb_1})$ and  $\bfS_{\tau_2}(\II_{\itLamb_2})$}
We use the following notation by Table \ref{tabel:notation} and the definition of integral normed (resp., Banach) module category in our paper.
\begin{enumerate}
  \item $\xi_r$ is an element in ${_{r}\!\II}=[c_r,d_r]_{\kk}$.

  \item $\II_{\itLamb_r}$, written as $\II_r$, is the subset of $\itLamb_r$ defined as $ \displaystyle\sum_{i=1}^n [c_r,d_r]_{\kk} b_i$.

  \item $\id_{\II_r}$, written as $\id_r$ for simplification, is the function $\II_r \to \{1\}$.

  \item $\tau_r: \itLamb_r\to\kk$ is a given homomorphism of $\kk$-algebras.

  \item $\mu_{{_{r}\!\II}}$ is the measure defined on the $\sigma$-algebra $\Sigma({_{r}\!\II})$ generated by
    \begin{center}
      $\{(c_r',d_r')_{\kk}, [c_r',d_r')_{\kk}, (c_r',d_r']_{\kk}, [c_r',d_r']_{\kk} \mid c_r\preceq c_r'\preceq d_r'\preceq d_r\}$.
    \end{center}
  \item $\mu_r$ is the measure, induced by $\mu_{\II_r}$, defined on the $\sigma$-algebra $\Sigma(\II_r)$ generated by $ \displaystyle \sum_{i=1}^n I_{ri} b_i$,
      where, for any $1\le i\le {n_r}$, $I_{ri}$ is one of the forms
      \begin{center}
        $(c_{ri},d_{ri})_{\kk}$, $(c_{ri},d_{ri}]_{\kk}$, $[c_{ri},d_{ri})_{\kk}$ and $[c_{ri},d_{ri}]_{\kk}$,
      \end{center}
      where $c_r\preceq c_{ri} \preceq d_{ri} \preceq d_r$.


  \item $\Thomo{r}: \sfct{r}_u \to \kk$ is the morphism defined as (\ref{def:Tu}).

  \item $\mhomo{r}:\kk^{\oplus_p 2^n} \to \kk$ is the $\itLamb_r$-homomorphism defined by the same method of the definition of (\ref{def:mhomo}).
\end{enumerate}
Then we can define $\scrN^p_r$ and $\scrA^p_r$ for each $\itLamb_r$, and,
by Theorem \ref{thm:LLHZpre-2} (2), there is an object in $\scrA^p_r$, which is of the form $(\kk, \mu_r(\II_r), \mhomo{r})$,
such that the homomorphism space $\Hom_{\scrN^p_r}((\bfS_{\tau_r}(\II_r), \id_r, \gamma_r), (\kk, \mu_r(\II_r), \mhomo{r}))$ contains a unique morphism $T_{(\kk, 1, \mhomo{r})}$ which can be extended to the unique morphism
\[\w{T}_{(\kk, \mu_r(\II_r), \mhomo{r})}: f \mapsto \w{T}_{(\kk, \mu_r(\II_r), \mhomo{r})}(f),\]
in $\Hom_{\scrA^p_r}((\w{\bfS_{\tau_r}(\II_r)}, \id_r, \w{\gamma}_r), (\kk, \mu_r(\II_r), \mhomo{r}))$.
In particular, if $p=1$, the above morphism is denoted by
\[(\scrA^1_r)\int_{\II_r} \cdot\dd\mu_r\]
in this paper.

Notice that $\II_r$ is a subset of $\itLamb_r$, we obtain a natural embedding $\emb_r: \II_r \to \itLamb_r$.
Thus, we get the following diagram
\[
\xymatrix{
  \II_1 \ar[r]^{\emb_1} \ar[d]_{\iso|_{\II_1}}
& \itLamb_1 \ar[d]^{\iso}  \\
  \II_2 \ar[r]_{\emb_2}
& \itLamb_2
}
\]
commutes. If $\iso$ is an injection (resp. a bijection), then it has a left inverse $\isoinv$ (resp. an inverse $\isoinv=\iso^{-1}$).
Therefore, $\isoinv|_{\II_2}\circled\iso = \id_1$ (resp.  $(\iso|_{\II_1})^{-1} = (\iso^{-1})|_{\II_2} = \isoinv|_{\II_2}$),
and any function $f:\II_1 \to \kk$ defined on $\II_1$ induces the function defined on $\II_2$ satisfying
\[ f \circled \isoinv = f \circled \isoinv|_{\II_2}: \ \
   \II_2
       \mathop{\longrightarrow}\limits^{\isoinv|_{\II_2}}
   \II_1
       \mathop{\longrightarrow}\limits^{f}
   \kk,
   \ \
x \mapsto f(\isoinv(x)). \]
Thus, if $\iso$ preserve step function, that is, $f\circled\iso$ is a step function lying in $\bfS_{\tau_2}(\II_2)$ for any $f\in\bfS_{\tau_1}(\II_1)$, then we obtain two correspondings
\[
\ol{\iso}: \bfS_{\tau_1}(\II_1) \to \bfS_{\tau_2}(\II_2), f \mapsto f\circled\isoinv|_{\im(\iso)}
\ \text{ and }\
\ol{\isoinv}: \bfS_{\tau_2}(\II_2) \to \bfS_{\tau_1}(\II_1), g \mapsto g\circled\iso.
\]

\begin{lemma} \label{lemma:inviso-iso}
The characteristics of mappings induced by the completions of the spaces $\bfS_{\tau_1}(\II_1)$ and $\bfS_{\tau_2}(\II_2)$ are as follows:
\begin{itemize}
  \item[\rm(1)] If $\iso$ is an injection, then $\w{\ol{\isoinv}}$, the map given by the completion $\w{\bfS_{\tau_2}(\II_2)}$ of $\bfS_{\tau_2}(\II_2)$, is also an injection.
      In this case, the left inverse $\isoinv$ of $\iso$ induces the left inverse $\ol{\isoinv}$ of $\ol{\iso}$
      and the left inverse $\w{\ol{\isoinv}}$ of $\w{\ol{\iso}}$.
  \item[\rm(2)] If $\iso$ is a bijection, then so is $\isoinv$.
    Let $\w{\ol{\iso}}$ and $\w{\ol{\isoinv}}$ are bijections given by the completions $\w{\bfS_{\tau_1}(\II_1)}$ and $\w{\bfS_{\tau_2}(\II_2)}$ of $\bfS_{\tau_1}(\II_1)$ and $\bfS_{\tau_2}(\II_2)$, respectively.
    Then $\ol{\iso}^{-1}=\ol{\isoinv}$ and $\w{\ol{\iso}}^{-1}=\w{\ol{\isoinv}}$.
\end{itemize}
\end{lemma}

\begin{proof}
We only prove (1), the proof of (2) is similar.
Let $f:\itLamb_1\to\kk$ be an arbitrary function in $\bfS_{\tau_1}(\II_1)$,
then $\ol{\isoinv}\circled\ol{\iso}(f) = f\circled\isoinv\circled\iso = f$,
that is, $\ol{\isoinv}\circled\ol{\iso}$ is the identity mapping defined on $\bfS_{\tau_1}(\II_1)$.
This induces that the left inverse of $\w{\ol{\iso}}$ is $\w{\ol{\isoinv}}$, naturally.
\end{proof}

\subsubsection{\sectcolor Measure preserving injection/bijection}

\begin{definition} \rm \label{def:fct}
We call $\iso$ is a {\defines measure preserving injection {\rm(}resp. bijection{\rm)}}, if it is an injection {\rm(}resp. a bijection{\rm)} such that the following two conditions hold.
\begin{itemize}
  \item[(1)] For any subset $S_1$ of $\itLamb_1$ lying in $\Sigma(\II_1)$,
    the image $S_2=\mathrm{Im}(\iso|_{S_1})$ of the restriction $\iso|_{S_1}:S_1\to S_2$ of $\iso$ is also a subset of $\itLamb_2$ lying in $\Sigma(\II_2)$.
  \item[(2)] There exists a function $\fct:\RR \to \RR$ such that the diagram
\begin{align}\label{comm.diagr in def:fct}
\hspace{-2cm}
 \xymatrix{
  \Sigma(\II_1) \ar[r]^{\mu_1} \ar[d]_{S \mapsto \im(\iso|_S)=\iso|_S(S)}
& \RR \ar[d]^{\fct}
  \\
  \Sigma(\II_2) \ar[r]_{\mu_2}
& \RR
}
\end{align}
commutes, that is, the equation
\begin{align}\label{comm.diagr.eq in def:fct}
(\fct\circled\mu_1)(S_1) = (\mu_2\circled\omega|_{S_1})(S_1)\ (=\mu_2(S_2))
\end{align}
holds for all $S_1\in\Sigma(\II_1)$.
\end{itemize}
\end{definition}

\begin{lemma} \label{lemma:measure}
Keeping the notations from Definition \ref{def:fct}, if $\iso$ preserves measure, then $\isoinv$ also preserve measure, and $\fct\circled\mu_1\circled\isoinv|_{\im(\iso)}$ is a measure $($note that in the case for $\iso$ being bijective, we have $\isoinv|_{\im(\iso)}=\isoinv$$)$.
\end{lemma}

\begin{proof}
We only prove this lemma in the case for $\iso$ being injective, the case for $\iso$ being bijective is similar.

Let $\iso_{\mathrm{L}}^{-1}$ be the left inverse $\isoinv$ of $\iso$.
For any subset $S_2$ of $\itLamb_2$ lying in $\Sigma(\II_2)$, we have $\iso_{\mathrm{L}}^{-1}|_{\im(\iso)}(S_2) = \isoinv|_{\im(\iso)}(S_2) \in \Sigma(\II_1)$ by $\iso$ preserving measure. Then we have
\[(\fct\circled\mu_1\circled\isoinv|_{\im(\iso)})(S_2) = \fct(\mu_1(\isoinv(S_2))) = \fct(\mu_1(\iso_{\mathrm{L}}^{-1}(S_2)))\]
which equals to $\mu_2(S_2)$ by (\ref{comm.diagr in def:fct}).
Thus, $\fct\circled\mu_1\circled\isoinv|_{\im(\iso)}=\mu_2$.
Since $\mu_2$ is a measure, so is $\fct\circ\mu_1\circ\isoinv|_{\im(\iso)}$.
\end{proof}

\subsection{\sectcolor The relationships of two integrals} \label{subsect:iso}

Now we provide the relationships of the integrals of functions defined on $\itLamb_1$ and that of function defined on $\itLamb_2$.

\subsubsection{\sectcolor In the case of $\iso$ being an
injection/bijection} \label{subsubsect:bijection}

\begin{theorem} \label{thm:main1}
Let $f$ and $g$ be two functions lying in $\w{\bfS_{\tau_1}(\II_1)}$ and $\w{\bfS_{\tau_2}(\II_2)}$ such that $g=\ol{\iso}(f)=f\circled\isoinv$, respectively.
Then $\fct\circled\mu_1$ is a measure defined on $\Sigma(\II_1)$ and
\begin{align}\label{formula:main1}
  \w{T}_{(\kk, \fct(\mu_1(\II_1)), \ \mhomo{1})}(f)
= (\scrA^1_1)\int_{\II_1} f \dd (\fct\circled\mu_1)
= (\scrA^1_2)\int_{\II_2} g \dd\mu_2
= \w{T}_{(\kk,\ \mu_2(\II_2), \ \mhomo{2})}(g).
\end{align}
\end{theorem}

\begin{proof}
The equation (\ref{def:fct}) yields $\fct(\mu_1(S))=\mu_2(\iso(S))\ge 0$ holds for all $S\in\Sigma(\II_1)$
and $\fct(\mu_1(\varnothing))=\mu_2(\iso(\varnothing)) = \mu_2(\varnothing) = 0$.
Now we show the countable additivity of $\fct\circled\mu_1$.
Take $\{X_i\}_{i\in\NN}$ is a family of subsets of $\itLamb_1$ lying in $\Sigma(\II_1)$ with $X_i\cap X_j=\varnothing$ ($i\ne j$).
Then we have
\begin{align*}
 \fct\bigg(\mu_1\bigg(\bigcup_{i=1}^{+\infty} X_i\bigg)\bigg)
& \mathop{=}\limits^{\spadesuit} \mu_2\bigg(\iso\bigg(\bigcup_{i=1}^{+\infty} X_i\bigg)\bigg)
\mathop{=}\limits^{\clubsuit} \mu_2\bigg(\bigcup_{i=1}^{+\infty} \iso(X_i)\bigg) \\
& \mathop{=}\limits^{\heart} \sum_{i=1}^{+\infty} \mu_2(\iso(X_i))
\mathop{=}\limits^{\spadesuit} \sum_{i=1}^{+\infty} \fct(\mu_1(X_i))
\end{align*}
as required, where two $\spadesuit$ are given by (\ref{def:fct}), $\clubsuit$ is given by $X_i\cap X_j=\varnothing$ ($i\ne j$), and $\heart$ is given by the injectivity of $\iso$ the countable additivity of the measure $\mu_2$.

Next, we prove (\ref{formula:main1}). Notice that we have a Cauchy sequence
$\{f_i\}_{i\in\NN} = \{f_i:\sfct{1}_{u_i} \to \kk\}_{i\in\NN}$ in $\bigcup\limits_{u\in\NN}\sfct{1}_u$.
Thus, by Lemma \ref{lemma:inviso-iso} (1), have the Cauchy sequence
\begin{center}
  $\{g_i\}_{i\in\NN} = \{g_i=f_i\circled\isoinv:\sfct{2}_{u_i} \to \kk\}_{i\in\NN}$
\end{center}
in $\bigcup\limits_{u\in\NN}\sfct{2}_u$ provided by $\{f_i\}_{i\in\NN}$
such that
\[ f = \underleftarrow{\lim} f_i, \ \
   g = \underleftarrow{\lim} g_i \text{ \ \ and \ \ }
   (\scrA^1_2)\int_{\II_2} g\dd\mu_2 = \underleftarrow{\lim}\ \Thomo{2}_{u_i}(g_i)
\]
hold, see (\ref{analy.lim}).
Then, assume that, for each $i\in\NN$, $f_i$ which is of the form $\sum_{j=1}^{t_i} k_{ij}\id_{I_{ij}}$ ($k_{ij}\in\kk$, $I_{ij}\cap I_{ij'}=\varnothing$ holds for all $1\le j\ne j'\le t_i$), we have
\begin{align*}
  (\scrA^1_2)\int_{\II_2} g\dd\mu_2
& = (\scrA^1_2)\int_{\II_2} f \circled \isoinv \dd\mu_2
  = \underleftarrow{\lim}\ \Thomo{2}_{u_i}(f_i\isoinv) \\
& = \underleftarrow{\lim}\ \Thomo{2}_{u_i}\bigg(\sum_{j=1}^{t_i} k_{ij} (\id_{I_{ij}}\circled\isoinv)\bigg) \\
& = \underleftarrow{\lim}\ \sum_{j=1}^{t_i} k_{ij} \
    \Thomo{2}_{u_i}(\id_{I_{ij}} \circled \text{id}_{\II_1}|_{I_{ij}} \circled\isoinv|_{\im(\iso)}) \\
& = \underleftarrow{\lim}\ \sum_{j=1}^{t_i} k_{ij}\ \Thomo{2}_{u_i}(\id_{\iso(I_{ij})}) \\
& = \underleftarrow{\lim}\ \sum_{j=1}^{t_i} k_{ij}\mu_2(\iso(I_{ij})).
\end{align*}
Thus, by $\mu_2 = \fct\circled\mu_1\circled(\iso|_{S_1})^{-1} = \fct\circled\mu_1\circled\isoinv|_{S_2}$
and $\iso(I_{ij})\subseteq S_2$, we obtain
\begin{align*}
    (\scrA^1_2)\int_{\II_2} (f \circled \isoinv) \dd\mu_2
& = \underleftarrow{\lim}\ \sum_{j=1}^{t_i} k_{ij}\mu_2(I_{ij}) \\
& = \underleftarrow{\lim}\ \sum_{j=1}^{t_i} k_{ij}\fct\big(\mu_1(\isoinv|_{\iso(I_{ij})}(\iso(I_{ij})))\big) \\
& \mathop{=}\limits^{\diamo}\
    (\scrA^1_1)\int_{\II_1} f \dd (\fct\circled\mu_1\circled\isoinv|_{\im(\iso)}\circled\iso) \\
& = (\scrA^1_1)\int_{\II_1} f \dd (\fct\circled\mu_1)
\end{align*}
as required, where $\iso$ (resp. $\isoinv$) can be seen as a map sending each set $S$ lying in the $\sigma$-algebra $\Sigma(\II_1)$ (resp. $\Sigma(\II_2)$) to the set $\iso|_S(S)$ (resp. $\isoinv|_S(S)$) belong to $\Sigma(\II_2)$ (resp. $\Sigma(\II_1)$), 
and, by Lemma \ref{lemma:measure}, $\diamo$ holds since $\fct\circled\mu_1\circled\isoinv|_{\im(\iso)}$ is a measure defined on $\Sigma(\II_2)$ (and so is $\fct \circled \mu_1 \circled \isoinv|_{\im(\iso)} \circled \iso$ $=$ $\fct\circled\mu_1$).
\end{proof}

\subsubsection{\sectcolor In the case of $\iso$ being an isomorphism} \label{subsubsect:iso}
Now, we assume that $\iso: \itLamb_1\to\itLamb_2$ is an isomorphism of algebras in this section.
Then, for each $r\in\{1,2\}$, there is a basis of $\itLamb_r$, written as $B_{\itLamb_r}=\{b_{r1}, \cdots, b_{rn_r}\}$, such that, for all $1\le i,j\le n_1=n_2$, the following statements hold.
\begin{itemize}
  \item[(1)] $\iso(b_{1i}) = b_{2i}$;
  \item[(2)] $\iso(b_{1i}b_{1j}) = b_{2i}b_{2j}$.
\end{itemize}
The correspondence $\iso$ can be written as the standard form
\[ \itLamb_1
   \ni \sum_{i=1}^n k_ib_{1i}
   \
   \begin{smallmatrix}
   \To{\iso} \\
   \oT{\isoinv}
   \end{smallmatrix}
   \
   \sum_{i=1}^n k_ib_{2i}
   \in \itLamb_2. \]

If $\mu_r$ is the measure defined on $\Sigma(\II_r)$ induced by the measure $\mu_{{_{r}\!\II}}: \Sigma({_{r}\!\II}) \to \RR$, where ${_{r}\!\II}=[c_r,d_r]_{\kk}$ and $\II_r = \prod_{i=1}^n [c_r,d_r]_{\kk}b_{ri}$,
then we can not \checks{discriminate} between the measures $\mu_1$ and $\mu_2$ in the case of ${_{1}\!\II} = {_{2}\!\II}$.
Thus, one can check that $\iso$ is a bijection preserve measure. We obtain the following result.

\begin{corollary} \label{coro:main2}
Let $f$ and $g$ be two functions lying in $\w{\bfS_{\tau_1}(\II_1)}$ and $\w{\bfS_{\tau_2}(\II_2)}$ such that $g=\ol{\iso}(f)$, respectively. If $\iso$ is an isomorphism, then
\[ (\scrA^1_1)\int_{\II_1} f \dd(\fct\circled\mu) = (\scrA^1_2)\int_{\II_2} g \dd\mu \]
\end{corollary}

\begin{proof}
It is the direct corollary of Theorem \ref{thm:main1}.
\end{proof}

Now we provide some examples for Corollary \ref{coro:main2}.

\begin{example} \label{examp} \rm
Recall that a {\defines quiver} is a pair $(\Q_0, \Q_1)$ with two functions $\s:\Q_1\to\Q_0$ and $\t:\Q_1\to\Q_0$, written as $\Q$, where $\Q_0$ is called the {\defines vertex set} whose elements are called {\defines vertices};
$\Q_1$ is called the {\defines arrow set} whose elements are called {\defines arrows},
and $\s$ and $\t$ are maps sending each arrow $\alpha\in\Q_1$ to its starting point and ending point, respectively.
Then the $\kk$-linear space $\mathrm{span}_{\kk}\{\wp \mid \wp \text{ is a path on } \Q\}$ is a $\kk$-algebra,
where the multiplication is induced by the composition of paths, that is, for any two paths $\wp_1$ and $\wp_2$,
if $\t(\wp_1)=\s(\wp_2)$, then the product of $\wp_1$ and $\wp_2$ is defined as the composition $\wp_1\wp_2$;
otherwise, i.e., in the case of $\t(\wp_1)\ne\s(\wp_2)$, we define the product of $\wp_1$ and $\wp_2$ is zero,
cf. \cite[Chap II]{ASS2006}. Respectively, we provide two algebras $A$ and $B$ shown in the item (1) and (2) of this Example, and show that $A\cong B$, and, moreover, provide two integrals of functions defined on $A$ and $B$.

(1) Let $A=\kk\Q/\I$ be a finite-dimensional algebra given by $\kk=\RR$, $\Q=\xymatrix{ 1 \ar@(ul,dl)_{a} \ar[r]_{b} & 2 }$ and the ideal $\I=\langle a-a^2 \rangle$. Then
\begin{align*}
 A & = \kk e_1 + \kk e_2 + \kk a + \kk b + \kk ab + \langle a^2-a \rangle
\end{align*}
Take $\II=[0,1]$, then we have the one-to-one correspondence
\[ \II_A = [0,1]_A \mathop{\longleftrightarrow}\limits^{1-1}
   [0,1]e_1 + [0,1]e_2 + [0,1]a + [0,1]b + [0,1]ab
=: [0,1]^{\dag}.\]
Now, define $f: \II_A \to \RR$ induced by the map $[0,1]^{\dag} \to \RR$ sending each element $k_1e_1 + k_2e_2 + k_3a + k_4b + k_5ab$ ($k_i\in\kk$, $1\le i\le 5$) to $k_1+k_2+k_4$.
If $\scrA_A^1$ satisfies \checks{(L1), (L3)--(L6)}, then we have
\begin{align} \label{examp:int 1}
  (\scrA_A^1) \int_{\II_A} f \dd\mu
 = \iiint_{[0,1]^{\times 3}} (k_1+k_2+k_4)\dd k_1 \dd k_2 \dd k_4
 = \frac{3}{2}.
\end{align}

(2) Let $B=\kk\mathcal{P}$ be a finite-dimensional algebra given by $\kk=\RR$, $\mathcal{P}=\xymatrix{ 1\ar[r]^{\alpha} & 2 & 3\ar[l]_{\beta}}$. Then
\begin{align*}
 B = \kk\varepsilon_1 + \kk\varepsilon_2 + \kk\varepsilon_3 + \kk \alpha + \kk \beta.
\end{align*}
Notice that we have the isomorphism $\iso: A \mathop{\longrightarrow}\limits^{\cong} B$ which is defined by
\begin{align*}
         & k_1e_1 + k_2e_2 + k_3a + k_4b + k_5ab \\
      =\ & (k_1+k_3)a + k_2e_2 + k_1(e_1-a) + (k_4+k_5)ab + k_4(e_1-a)b   \\
\mapsto\ & (k_1+k_3)\varepsilon_1 + k_2\varepsilon_2 + k_1\varepsilon_3 + (k_4+k_5)\alpha + k_4\beta,
\end{align*}
and its inverse $\iso^{-1}=\isoinv: B \to A$ is
\begin{align*}
        r_1\varepsilon_1 + r_2\varepsilon_2 + r_3\varepsilon_3 + r_4\alpha + r_5\beta
\mapsto r_3e_1 + r_2e_2 + (r_1-r_3)a + r_5b + (r_4-r_5)ab
\end{align*}
Take $\II=[0,1]$, then we have the one-to-one correspondence
\[ \II_B = [0,1]_B \mathop{\longleftrightarrow}\limits^{1-1}
   [0,1]\varepsilon_1 + [0,1]\varepsilon_2 + [0,1]\varepsilon_3 + [0,1]\alpha + [0,1]\beta; \]
the transformation of integration regions
\begin{align*}
 \isoinv(\II_B) &
= \{ r_3e_1 + r_2e_2 + (r_1-r_3)a + r_5b + (r_4-r_5)ab
\mid r_1, r_2, r_3, r_4, r_5 \in [0,1] \} \\
& \subseteq [0,1]e_1+[0,1]e_2+[-1,1]a+[0,1]b+[-1,1]ab \\
& =: I \ (\subseteq A)
\end{align*}
and the function $f\circled\isoinv: \II_B\to \RR$, the function given in the instance (1) of this example, is defined by
\[ r_1\varepsilon_1 + r_2\varepsilon_2 + r_3\varepsilon_3 + r_4\alpha + r_5\beta \mapsto r_3 + r_2 + r_5. \]
Thus, if $\scrA_B^1$ satisfies L-condition, then we have
\begin{align} \label{examp:int 2}
    (\scrA_B^1) \int_{\II_B} (f\circled\isoinv) \dd\mu \nonumber
& = \iiint_{[0,1]^{\times 3}} (r_3+r_2+r_5)\dd r_3 \dd r_2 \dd r_5 \\
& = \frac{3}{2} = (\scrA_A^1) \int_{\II_A} f\dd\mu,
\end{align}
here, $\fct=\id_{\RR}$. Then by (\ref{examp:int 1}) and (\ref{examp:int 2}), we have
\begin{align}\label{examp:int 3}
  (\scrA_B^1) \int_{\II_B} (f\circled\isoinv) \dd\mu
= (\scrA_A^1) \int_{I} f\dd(\fct\circled\mu)
\mathop{=}\limits^{\star} (\scrA_A^1) \int_{\II_A} f\dd(\fct\circled\mu),
\end{align}
where ``$\star$'' holds since $f=f|_{[0,1]e_1+[0,1]e_2+[0,1]b}$ can be seen as the restriction of the function
\begin{center}
$f: A \to \kk$, (resp. $f: I \to\kk$)

$x=k_1e_1+k_2e_2+k_3a+k_4b+k_5ab \mapsto
\begin{cases}
k_1+k_2+k_5, & x\in \II_A;  \\
0, & \text{otherwise}
\end{cases}$
\end{center}
defined on $A$ (resp. $I$).
\end{example}

\begin{remark} \rm
The isomorphism $\iso: A \to B$ given in Example \ref{examp} induces the following relationship
\begin{align*}
& f|_{\kk e_1+\kk e_2 + \kk b}:
  k_1e_1+k_2e_2+k_4b \mapsto k_1+k_2+k_4 \\
\ \mathop{\longrightarrow}\limits^{\isoinv} \
& f\circled\isoinv|_{\kk\varepsilon_2+\kk\varepsilon_3+\kk\beta}:
  r_2\varepsilon_2 + r_3\varepsilon_3 + r_5\beta \mapsto r_2+r_3+r_5
\end{align*}
which can be seen as a composition of a map sending $k_1e_1+k_2e_2+k_4b$ to $r_2\varepsilon_2 + r_3\varepsilon_3 + r_5\beta$ and a $\kk$-linear map $w$ sending $r_1\varepsilon_1+ r_2\varepsilon_2 + r_3\varepsilon_3+ r_4\alpha + r_5\beta$ to $r_2+r_3+r_5$.
In this case, as a non-rigorous perspective, the integration of the integral (\ref{examp:int 1}) can be understood as a composition of the vector valued integral
\[(\mathrm{B})\int_{W_{u,v}} (k_1e_1+k_2e_2+k_4b) \otimes_{\kk} \dd A\]
and the $\kk$-linear map $\phi$ sending each element $k_1e_1+k_2e_2+k_3a+k_4b+k_5ab$ to $k_1+k_2+k_3+k_4+k_5$ in the case for $u=0<v=1$, where:
\begin{itemize}
  \item $\displaystyle(\mathrm{B})\int$ is a Bochner integration; 
  \item $W_{u,v}=[u,v]e_1+[u,v]e_2+[-(v-u),v-u]a+[u,v]b+[-(v-u),v-u]ab$,
    here, $u<v$, and one can check $W_{0,1}=\isoinv(\II_B)$;
  \item $\dd A = (\dd k_1\cdot e_1)\wedge(\dd k_2\cdot e_2)\wedge(\dd k_3\cdot a)\wedge
     (\dd k_4\cdot b)\wedge(\dd k_5\cdot ab)$ (``$\wedge$'' represents exterior differential form);
  \item and the following equation
\begin{align*}
   & (\mathrm{B})\int_{W_{u,v}} (k_1e_1+k_2e_2+k_4b) \otimes_{\kk} \dd A \\
=\ &  (v-u)^4\cdot(\mathrm{L})\int_u^v k_1e_1\cdot \dd k_1 \cdot e_1 \\
   &+ 16(v-u)^4\cdot(\mathrm{L})\int_{u-v}^{v-u} k_1e_1\cdot \dd k_1 \cdot a \\
   &+ 16(v-u)^4\cdot(\mathrm{L})\int_{u-v}^{v-u} k_1e_1\cdot \dd k_1 \cdot ab \\
   &+ (v-u)^4\cdot(\mathrm{L})\int_u^v k_2e_2\cdot \dd k_2 \cdot e_2 \\
   &+ (v-u)^4\cdot(\mathrm{L})\int_u^v k_4b \cdot \dd k_4 \cdot e_2 \\
=\ & (v-u)^4\cdot\frac{v^2-u^2}{2}e_1 + (v-u)^4\cdot\frac{v^2-u^2}{2}e_2 + 0 a \\
   &+ (v-u)^4\cdot\frac{v^2-u^2}{2} b + 0 ab\ (\in A)
\end{align*}
  holds, where $\displaystyle (\mathrm{L})\int$ is a Lebesgue integration.
\end{itemize}
Thus,
  \[ (\mathrm{B})\int_{\isoinv(\II_B)} (k_1e_1+k_2e_2+k_4b) \otimes_{\kk} \dd A
  =  \frac{1}{2}e_1 + \frac{1}{2}e_2 + 0 a + \frac{1}{2} b + 0 ab. \]
Similarly, the integration of the integral (\ref{examp:int 2}) can be understood as a composition of the vector valued integral
\[(\mathrm{B})\int_{[u,v]_B} (r_2\varepsilon_2+r_3\varepsilon_3+r_5\beta) \otimes_{\kk}
  \det(\pmb{J})\dd B\]
and the $\kk$-linear map $\varphi$ sending each element $r_1\varepsilon_1 + r_2\varepsilon_2 + r_3\varepsilon_3 + r_4\alpha + r_5\beta$ to $r_1+r_2+r_3+r_4+r_5$ in the case for $u=0<v=1$, where:
\begin{itemize}
  \item $\pmb{J}$ is the Jacobi Matrix
     \[\pmb{J} = \left(\begin{matrix}
      0 & 0 & 1 & 0 & 0 \\
      0 & 1 & 0 & 0 & 0 \\
      1 & 0 & -1& 0 & 0 \\
      0 & 0 & 0 & 0 & 1 \\
      0 & 0 & 0 & 1 & -1
     \end{matrix}\right) \]
     of $\isoinv: B\to A$;
  \item $\dd B = (\dd r_1 \cdot \varepsilon_1)\wedge (\dd r_2 \cdot \varepsilon_2)
          \wedge (\dd r_3 \cdot \varepsilon_3)\wedge (\dd r_4 \cdot \alpha)
          \wedge (\dd r_5 \cdot \beta)$;
  \item and the following equation
\begin{align*}
   &  (\mathrm{B})\int_{[u,v]_B} (r_2\varepsilon_2+r_3\varepsilon_3+r_5\beta) \otimes_{\kk}
      \det(\pmb{J})\dd B \\
=\ &  (v-u)^4\cdot(\mathrm{L})\int_{u,v}r_2\varepsilon_2 \cdot \dd r_2\cdot\varepsilon_2 \\
   &+ (v-u)^3\cdot(\mathrm{B})\iint_{(r_3,r_5)\in[u,v]^{\times 2}}
      r_3\varepsilon_3\cdot \dd r_3 \wedge \dd r_5\cdot \beta \\
   &+ (v-u)^3\cdot(\mathrm{B})\iint_{(r_2,r_5)\in[u,v]^{\times 2}}
      r_5\beta\cdot \dd r_5\wedge\dd r_2\cdot \varepsilon_2  \\
=\ & 0\varepsilon_1 + (v-u)^4\cdot\frac{(v^2-u^2)}{2}\varepsilon_2 + 0\varepsilon_3
    + 0\alpha + (v-u)^4(v^2-u^2)\beta
\end{align*}
  holds.
\end{itemize}
Thus,
\begin{align*}
    (\mathrm{B})\int_{[u,v]_B}
     (r_2\varepsilon_2+r_3\varepsilon_3+r_5\beta) \otimes_{\kk} \dd B
=   0\varepsilon_1 + \frac{1}{2}\varepsilon_2 + 0\varepsilon_3
  + 0\alpha + \beta \ (\in B)
\end{align*}
Therefore, (\ref{examp:int 3}) can be understood as the equation
\[ \phi\bigg((\mathrm{B})\int_{\isoinv(\II_B)} f\dd A\bigg)
= \frac{3}{2} = \varphi\bigg((\mathrm{B})\int_{\II_B} f\circled \isoinv\dd B\bigg) \]
about two Bochner intersections in Example \ref{examp}.

Notice that $\phi$ is induced by the sum $\sum_{v\in\Q_0} k_v$ of the augmentations $(A\to \kk e_v)_{v\in\Q_0}$ of $A$
and $\varphi$ is induced by the sum  $\sum_{v\in\mathcal{P}_0} r_v$ of the augmentations $(B\to \kk \varepsilon_v)_{v\in\mathcal{P}_0}$ of $B$.
Here, an {\defines augmentation} of an associative algebra $A$ over a field (or a commutative ring) $\kk$ is a $\kk$-algebra homomorphism $A\to\kk$;
and an algebra together with an augmentation is called an {\defines augmented algebra}, see \cite{LV2012}.
\end{remark}

\section{\sectcolor Applications: Stieltjes integrations} \label{sect:app}

We provide some applications for Theorem \ref{thm:main1} in this section.

\subsection{\sectcolor Lebesgue-Stieltjes integrations} \label{app:LSint}

Let $f:[\alpha,\beta] \to \RR$ be a non-negative function and bounded and $\dfct: [\alpha,\beta] \to \RR$ be a monotone non-decreasing and right-continuous, function where $\alpha, \beta \in \RR$.
Define $w((s,t]) = \dfct(t)-\dfct(s)$ and $w(\{\alpha\})=0$ (alternatively, the construction works for $\dfct$ left-continuous, and $w([s,t)) = \dfct(t)-\dfct(s)$ $w(\{\beta\})=0$).
By Carath\'{e}odory's extension theorem, there is a unique Borel measure $u_{\dfct}$ on $[\alpha,\beta]$ which agrees with $w$ on every interval.
The measure $u_{\dfct}$ arises from an outer measure (in fact, a metric outer measure) given by
\[u_{\dfct}(E) = \inf\left\{ \sum_i u_{\dfct}(I_i) \ \bigg|\  E \subseteq \bigcup_i I_i  \right\}\]
the infimum taken over all coverings of $E$ by countably many semiopen intervals.
In analysis, this measure is called the {\defines Lebesgue-Stieltjes measure} associated with $\dfct$.
Furthermore, the {\defines Lebesgue-Stieltjes integral}
\[(\mathrm{L\text{-}S})\int_{\alpha}^{\beta} f(x)\dd \dfct(x)\]
is defined as the Lebesgue integral of $f$ with respect to the measure $u_{\dfct}$ in the usual way.

\begin{example} \rm \label{ex:LSint}
Keep the notations from Definition \ref{def:fct}.
Take a monotone non-decreasing and right-continuous function $\varphi:[0,1] \to \RR$, $\itLamb_2=\RR$,
and the category $\scrA_2^p$ satisfying L-condition ($\II=[x_0, x_1]$ satisfies $x_0 = \min\limits_{x}\varphi(x) = \varphi(0)$ and $x_1 = \max\limits_{x}\varphi(x) = \varphi(1)$),
and let $\scrA_1^p$ satisfy the following conditions
\begin{itemize}
  \item[(1)] The conditions (L1)--(L5) in L-condition hold.
  \item[(2)] The measure $\mu_1$ is the Lebesgue-Stieltjes measure $u_{\varphi}$ associated with a function $\varphi$.
\end{itemize}
Then $\mu_1$ sending each interval $(a,b]$ to $\varphi(b)-\varphi(a)$ and $\mu_2$ is the Lebesgue measure defined on $\Sigma([x_0,x_1])$.
Consider the map $\Sigma([0,1]) \to \Sigma([x_0,x_1])$, $S\mapsto\iso|_{S}(S)$ induced by $\iso|_{(a,b]}((a,b]) = (\varphi(a),\varphi(b)]$ (i.e., induced by $\iso: x\mapsto\varphi(x)$),
then the diagram (\ref{comm.diagr in def:fct}) and the equation (\ref{comm.diagr.eq in def:fct}) are
 \[\xymatrix@C=2.5cm@R=2.5cm{
  \Sigma([0,1]) \ar[r]^{\mu_1=u_{\varphi}} \ar[d]_{S\mapsto\iso|_{S}(S)}
& \RR^{\ge 0} \ar@{=}[d]^{1_{\RR^{\ge 0}} \ ,  }
  \\
  \Sigma([x_0,x_1]) \ar[r]^{\mu_2}_{\text{(Lebesgue measure)}}
& \RR^{\ge 0}
}\]
and, for any $S\in\Sigma([0,1])$,
\[\mu_2\circled \iso|_S(S) = \mu_1(S) =  u_\varphi(S), \]
respectively.  Then (\ref{formula:main1}) yields that
\begin{align} \label{formula:cat-LSint}
  \w{T}_{(\RR, u_{\varphi}([0,1]), \mhomo{1})}(f)
 = \w{T}_{(\RR, \mu_2([x_1,x_0]), \mhomo{2})} (f\circled\isoinv)
\end{align}
for any $0\le f \in \w{\bfS_{\id_{\RR}}([0,1])}$, where $u_{\varphi}([0,1])=\varphi(1)-\varphi(0)$ and $\mu_2([x_0,x_1])=x_1-x_0$.

Notice that the left of (\ref{formula:cat-LSint}) is a Lebesgue-Stieltjes integral and the right of it is a Lebesgue integral in this example. In other words, we obtain
\[
 (\text{L-S})\int_0^1 f(x) \dd\varphi(x) = (\text{L}) \int_{x_0}^{x_1} f(\isoinv(x)) \dd\mu_2.
\]
\end{example}

\subsection{\sectcolor Riemann-Stieltjes integrations} \label{app:RSint}

Let $f$ be a real variable on the interval $[\alpha,\beta]$ with respect to another function $\varphi:\RR \to \RR$.
Take $P:$ $\alpha = t_0 < t_1 < \ldots < t_n = \beta$ a partition of $[\alpha,\beta]$, and, for any sequence $\{\eta_i\}_{i=1}^n$ with $\eta_i\in [t_i, t_{i+1}]$, define the {\defines Riemann-Stieltjes sum} by
\[ S(P, f, \varphi) = \sum_{i=0}^{n-1} f(\eta_i)(\varphi(t_{i+1})-\varphi(t_i)). \]
The {\defines Riemann-Stieltjes integral} of $f(x)$ respect to $\varphi(x)$ is the limit
\[ (\text{R-S})\int_{\alpha}^{\beta}  f(x) \dd\varphi(x)
:= \lim_{\lambda \to 0} S(P,f,\varphi), \]
where $\lambda = \max\limits_{0\le i< n}(t_{i+1}-t_i)$.

\begin{example} \rm \label{ex:RSint}
Keep the notations from Definition \ref{def:fct}. Now we provide an interpretation for Riemann-Stieltjes integration in the case of $\varphi: [0,1] \to \RR$ to be an injective absolutely continuous function.
Take $\itLamb_2=\RR$, and the category $\scrA_2^p$ satisfying L-condition ($\II=[x_0, x_1]$ satisfies $x_0 =  \min\limits_{x}\varphi(x)$ and $x_1 = \max\limits_{x}\varphi(x)$),
and let $\scrA_1^p$ satisfy the following conditions.
\begin{itemize}
  \item[(1)] The conditions (L1)--(L5) in L-condition hold.
  \item[(2)] $\mu_1$ is the measure induced by $[a,b] \mapsto \varphi(b)-\varphi(a)$ for any sufficiently small interval $[a,b] \subseteq [0,1]$.
\end{itemize}
Consider the map $\Sigma([0,1]) \to \Sigma([x_0,x_1])$, $S\mapsto\iso|_{S}(S)$ induced by $\iso: x\mapsto\varphi(x)$,
then, similar to (\ref{formula:cat-LSint}), we have
\begin{align}\label{formula:cat-RSint}
 \w{T}_{(\RR,\ u_1([0,1]),\ \mhomo{1})}(f)
   = \w{T}_{(\RR,\ \mu_2([x_0,x_1]),\ \mhomo{2})}(f\circled\isoinv)
\end{align}
for any $f \in \w{\bfS_{\id_{\RR}}([0,1])}$ by (\ref{comm.diagr.eq in def:fct}), where
$$\mu_1([0,1]) = \lim\limits_{\lambda\to 0} \sum\limits_{i=1}^{n-1} (\varphi(t_{i+1})-\varphi(t_i)), $$
$0=t_1<t_2<\ldots<t_n=1$ is an arbitrary partition of $[0,1]$, $\lambda=\max(t_{i+1}-t_i)$.

Notice that the left of (\ref{formula:cat-RSint}) is a Riemann-Stieltjes integral and the right of it is a Lebesgue integral in this example.
In other words, we obtain
\[(\text{R-S})\int_0^1 f(x) \dd\varphi(x)
= (\text{L}) \int_{x_0}^{x_1} f(\isoinv(x)) \dd\mu_2.
 \]
\end{example}

\subsection{\sectcolor Integration by substitution} \label{app:substitution}

Keep the notations from Definition \ref{def:fct}, and consider the categories $\scrA^p_1$ and $\scrA^p_2$ satisfies the following conditions:
\begin{itemize}
  \item[(1)] $\scrA^p_1$ and $\scrA^p_2$ satisfy L-conditions are integral Banach module categories of $\RR$, that is, $\scrA^p_1=\scrA^p_2=\scrA^1$ (take ${_{1}\!\II}=[0,1]$ and ${_{2}\!\II}=[\alpha,\beta]$ in this case);
  \item[(2)] $\fct: \RR\to \RR$ is a function such that $\fct|_{{_{1}\!\II}}: {_{1}\!\II}=[0,1] \to {_{2}\!\II}=[\alpha, \beta]$ is  continuous and bijective;
  \item[(3)] and $\iso=\fct$ (it is induced by the diagram (\ref{comm.diagr in def:fct})).
\end{itemize}
Then we have $\isoinv=\fct^{-1}$ is bijective and obtain
\begin{align}\label{Substitution1}
   (\mathrm{L}) \int_0^1 f(x) \dd(\fct\circled\mu)
 = (\mathrm{L}) \int_{\alpha}^{\beta} f(\isoinv (x)) \dd\mu
\end{align}
by Theorem \ref{thm:main1}, where $\mu_1$ and $\mu_2$, written as $\mu$ in the above formula for simplification, are Lebesgue measures defined on the $\sigma$-algebras $\Sigma([0,1])$ and $\Sigma([\alpha,\beta])$, respectively.
Furthermore, if $\fct$ is differentiable and $f$ is a step function in $E_u$ (or, generally, is a Riemann integrability function), then the left integral in the above equation is also a Riemann-Stieltjes integral, and furthermore, two integrals in the above formula are Riemann integrals, that is, we have
\begin{align}\label{Substitution2}
   (\text{R-S}) \int_0^1 f(x) \dd \fct(x)
 = (\text{R}) \int_0^1 f(x) \dd \fct(x)
 \mathop{=}\limits^{\spadesuit}
   (\text{R}) \int_{\alpha}^{\beta} f(\fct^{-1}(x)) \dd x.
\end{align}
The integral on the left side of the $\spadesuit$ can be seen as a Riemann integral obtain by applying the substitution $x=\fct(y)$ to the Riemann integral on the right side of $\spadesuit$.
The formulas (\ref{Substitution1}) and (\ref{Substitution2}) are called the {\defines substitution rules} of integral in analysis.

\subsection{\sectcolor Lebesgue integrations and Riemann integrations} \label{app:Lint and Rint}

By the works of the authors of \cite{LLHZpre}, the Lebesgue integration is a morphism in the category $\scrA^p$ satisfying L-condition.
Interestingly, if $\scrA^p_1$ and $\scrA^p_2$, the categories satisfying L-condition, coincide,
then, keep the notations from Definition \ref{def:fct}, ${_{1}\!\II}={_{2}\!\II}=[0,1]$, and $\iso=\mathrm{id}_{\RR}$ ($=\isoinv$),
we have $\scrA^p_1$ and $\scrA^p_2$ satisfy the conditions (1)--(3) in Subsection \ref{app:substitution} ($\alpha=0$, $\beta=1$), and, immediately, we have $\fct=\id_{\RR}: x\mapsto x$ for all $x\in\RR$ by the commutative diagram (\ref{comm.diagr in def:fct}).
In this case, the substitution rule (\ref{Substitution1}) is trivial and provide an interpretation for Lebesgue integration.

It is well-know that all Riemann integrability functions defined on $[0,1]$ are elements lying in $\w{\bfS_{\id_{\RR}}([0,1])}$.
Let $R([0,1])$, as a subset of $\w{\bfS_{\id_{\RR}}([0,1])}$, be the set of all Riemann integrability functions,
then, naturally, we have the restriction
\[\gamma_{\frac{1}{2}}|_{R([0,1])^{\oplus_1 2}}: R([0,1])^{\oplus_1 2} \to R([0,1])\]
of $\gamma_{\frac{1}{2}}$, and obtain a triple $(R([0,1]), \id_{[0,1]}, \gamma_{\frac{1}{2}}|_{R([0,1])^{\oplus_1 2}})$, as a object in $\scrN^1$, which is an subobject of $(\bfS_{\id_{\RR}}([0,1]), \id_{[0,1]}, \gamma_{\frac{1}{2}})$.
Then the restriction
\[ \w{T}_{(\RR,1,m)}|_{R([0,1])}
   \in
   \Hom_{\scrN^1}
     (
       (R([0,1]), \id_{[0,1]}, \gamma_{\frac{1}{2}}|_{R([0,1])^{\oplus_1 2}}),
       (\RR, 1, m)
     ) \]
of the morphism
\[ \w{T}_{(\RR,1,m)}
   \in
   \Hom_{\scrA^1}
     (
       (\bfS_{\id_{\RR}}([0,1]), \id, \gamma_{\frac{1}{2}}),
       (\RR,1,m)
     )\]
sending each function in $R([0,1])$ to its Riemann integral.

Therefore, continuing to discuss the first paragraph of this subsection,
if $f\in R([0,1])$, then (\ref{Substitution1}) and (\ref{Substitution2}) are trivial and coincident.
Two equations provide an interpretation for Riemann integrations.


\section*{Acknowledgements}

Yu-Zhe Liu is supported by the National Natural Science Foundation of China (Grant No. 12171207),
Guizhou Provincial Basic Research Program (Natural Science) (Grant No. ZK[2024]YiBan066)
and Scientific Research Foundation of Guizhou University (Grant Nos. [2023]16, [2022]53, [2022]65).

Hanpeng Gao is supported by the National Natural Science Foundation of China (Grant No. 12301041).

Shengda Liu is supported by the National Natural Science Foundation of China (Grant No. 62203442).







\def\cprime{$'$}

\end{document}